\documentclass[a4paper,10pt,twoside,centertags,leqno]{article}
\usepackage{a4wide,amsmath,bbm,mathtools,mathrsfs,amsfonts,fancyhdr,color, ifthen, graphicx,stmaryrd, longtable,amssymb,latexsym,amsxtra,amscd,epic,array}
\usepackage[numbers,sort&compress]{natbib}
\usepackage[amsmath]{ntheorem}
\usepackage{hyperref}
\hypersetup{colorlinks=false}
\usepackage{hypernat}
\usepackage[latin1]{inputenc}
\newcommand{\MN}{\mathbb N}
\newcommand{\MZ}{\mathbb Z}
\newcommand{\MQ}{\mathbb Q}
\newcommand{\MR}{\mathbb R}
\newcommand{\MP}{\mathbb P}

\newcommand{\hg}{\ \hat{=}\ }
\newcommand{\qed}{\hfill$\Box$}

\DeclareMathOperator{\GL}{GL}
\DeclareMathOperator{\SL}{SL}
\DeclareMathOperator{\Sp}{Sp}

\DeclareMathOperator{\id}{id}
\DeclareMathOperator{\Wr}{Wr}

\DeclareMathOperator{\SO}{SO}

\DeclareMathOperator{\Sym}{Sym}
\DeclareMathOperator{\rank}{rank}
\DeclareMathOperator{\Sol}{Sol}
\DeclareMathOperator{\Gal}{Gal}

\DeclareMathOperator{\CC}{\mathbb{C}}
\DeclareMathOperator{\NN}{\mathbb{N}}

\DeclareMathOperator{\QQ}{\mathbb{Q}}
\DeclareMathOperator{\ZZ}{\mathbb{Z}}

\def\S{{\bf S}}

\def\Sp{{\rm Sp}}
\def\GO{{\rm GO}}
\def\rk{{\rm rk}}

\def\B{{\cal B}}

\def\MH{{\rm MH}}

\def\Sym{{\rm{Sym}}}

\newcommand{\si}{\sigma}

\newcommand{\ga}{\Gamma}

\newcommand{\Mod}{ {\rm Mod} }

\renewcommand{\theta}{{\vartheta}}
\def\K{{\cal K}}

\def\la{\lambda}

\def\GL{{\rm GL}}
\def\SL{{\rm SL}}

\def\im{{\rm im}}
\def\id{{\rm id}}

\def\S{{\bf S}}

\def\su{{\underline{s}}}

\def\T{{\bf T}}
\def\L{{L}}

\def\J{{\bf J}}
\def\Wr{{\rm Wr}}
\def\K{{\cal K}}

\def\MT{{\rm MT}}
\def\MC{{\rm MC}}

\newcommand{\ten}{{\otimes}}

\newcommand{\tr}{{\mbox{\rm {\small tr}}}}

\newcommand{\Jordan}{{\bf J}}

\newcommand{\diag}{{\rm diag}}

\newcommand{\La}{{\bf{\Lambda}}}

\newtheorem{defin}{Definition}[section]
\newtheorem{thm}[defin]{Theorem}
\newtheorem{lemma}[defin]{Lemma}
\newtheorem{cor}[defin]{Corollary}
\newtheorem{prop}[defin]{Proposition}
\newtheorem{rem}[defin]{Remark}
\newtheorem{ex}[defin]{Example}

\newtheorem*{conv}{Convention}
\newtheorem*{ethm}{Existence Theorem}
\newtheorem*{conj}{Conjecture}

\makeatletter
\renewcommand{\subsection}{\@startsection{subsection}{2}%
{\z@}{-3.25ex plus -1ex minus-.2ex}{-1em}{\bf}} \makeatother

\begin{document}

\thispagestyle{empty}
\begin{center}
{\LARGE\bf On symplectically rigid local systems of rank four and Calabi-Yau operators}

\vspace{.25in} 
{\large {\sc Michael Bogner and 
Stefan Reiter}}
\footnote{Both authors were partially supported by the SFB/TR $45$ 'Periods, Moduli Spaces and Arithmetic of Algebraic Varieties' of the DFG.}\\
E-mail:
{\tt    mbogner@mathematik.uni-mainz.de,
 reiters@uni-mainz.de} 
\end{center}

\begin{abstract}
We classify all $\Sp_4(\mathbb{C})$-rigid, quasi-unipotent local systems and show that all of them have geometric origin. Furthermore, we investigate which of those having a maximal unipotent element are induced by fourth order Calabi-Yau operators. Via this approach, we reconstruct all known Calabi-Yau operators inducing a $\Sp_4(\mathbb{C})$-rigid monodromy tuple and obtain closed formulae for special solutions of them.
\end{abstract}

\section{Introduction}

Differential operators \textit{of geometric origin} are proposed to describe periods of families of complex algebraic varieties and have been studied quite extensively during the last fifty years. A special class of such operators are \textit{fourth order differential Calabi-Yau operators} which are related to families of Calabi-Yau threefolds having a large complex structure limit and $h^{2,1}=1$. A conjectural characterization of those operators from a purely differential algebraic point of view, together with a list of most of the known examples is stated in \cite{AESZ}. The majority of those operators is not constructed from a geometric situation, as only very few examples of this type are known at the moment. Thus it is natural to ask, which of the operators really are of geometric origin and what would be a geometric realization.

It is quite challenging to decide whether a given differential operator has geometric origin or not. The first order ones are exactly those, which have a non trivial algebraic solution, see e.g. \cite{Simpson2}. Furthermore, as observed by Y. Andr\'e in \cite[Chapter II]{Andre}, the class of geometric differential operators is preserved by a multitude of constructions as taking subquotients, direct sums, tensor products and Hadamard products. We call an operator which can be obtained in this way \textit{geometrically constructible}. 
An appropriate method to check whether an operator is geometrically constructible or not is provided by the following investigation of local solutions.

Given a differential operator $L$ of degree $n$ with coefficients in $\mathbb{C}(z)$ and singular locus $S$, a classical theorem due to Cauchy states that for each $x\in\mathbb{P}^1\setminus S$ we find a basis $F=\{f_1,\dots,f_n\}$ of the $n$-dimensional $\mathbb{C}$-vectorspace $\Sol(L)_{x}=\{L(f)=0\mid f \textrm{ is holomorphic in some disc around } x\}$. 
If we chose a closed path $\gamma$ starting at $x$, analytic continuation of $F$ around $\gamma$ yields a different basis $\tilde{F}$ of $\Sol(L)_{x}$. The change from $F$ to $\tilde{F}$ only depends on the homotopy class of $\gamma$, which reflects the elements of $S$ encircled by $\gamma$. The translation of Cauchy's theorem into $20$th century language thus states the following: the operator $L$ induces a local system $\mathbb{L}$ of rank $n$ on $\mathbb{P}^1\setminus S$ via
\[\mathbb{L}(U):=\{f\in\mathcal{O}_{\mathbb{P}^1\setminus S}(U)\mid L(f)=0\}.\] Furthermore, with respect to an arbitrary base point  $x_0\in\mathbb{P}^1\setminus S$ this local system naturally induces a representation \[\rho_{\mathbb{L}}\colon \pi_1\left(\mathbb{P}^1\setminus S,x_0\right)\to \GL(\mathbb{L}_{x_0})\] of $\pi_1\left(\mathbb{P}^1\setminus S,x_0\right)$, the so called \textit{monodromy representation}. Its image is called the \textit{monodromy group} associated to $L$. We may chose a set of generators $(\gamma_s)_{s\in S}\subset \pi_1(\mathbb{P}^1\setminus S,x_0)$, whose elements are just simple loops $\gamma_s$ around each $s\in S$. As $S$ is finite, it can be equipped with an ordering $I$ such that \[
\prod_{i\in I}\gamma_{s_i}=1\in\pi_1\left(\mathbb{P}^1\setminus S,x_0\right)\] holds. Thus the monodromy group is completely determined by the tuple \[(T_{s_i})_{i\in I}:=\left(\rho_{\mathbb{L}}(\gamma_{s_i})\right)_{i\in I}\] of linear maps, which fulfill $\prod_{i\in I}T_{s_i}=\id_{\mathbb{L}_{x_0}}$. 
This tuple $(T_s)_{s\in S}$ is called the \textit{monodromy tuple} associated to $L$ and represents the effect of analytic continuation of holomorphic solutions near $x$ around each singularity of $L$. We call a monodromy tuple to be of geometric origin, if it is induced by a differential operator of geometric origin.

The constructions preserving the geometric origin of an operator have counterparts on the level fuchsian systems and monodromy tuples, see \cite{katz96} and \cite{DR07}. Furthermore, taking tensor- and middle Hadamard products with rank one tuples of geometric origin is an invertible operation. Thus a tuple is of geometric origin, if we can produce a tuple of geometric origin out of it, using those invertible operations. 

As shown by N. Katz in \cite{katz96}, a subclass of monodromy tuples of geometric origin are the \textit{linearly rigid} ones, i.e. those, whose elements are quasi-unipotent, generate an irreducible subgroup in $\GL_n(\mathbb{C})$ and which, are up to simultaneous conjugation, completely determined by the Jordan forms of its elements. In particular, Katz shows that each tuple of this type can be reduced to a geometric tuple of rank one by a sequence of invertible operations introduced above. The most prominent examples of linearly rigid tuples are those induced by generalized hypergeometric differential equations and where studied by Levelt \cite{Levelt} and Beukers and Heckmann \cite{BH90} for instance.

One can extend the notion of rigidity from  $\GL_n(\mathbb{C})$ to any reductive complex algebraic group, but then a reduction a la Katz generally fails. Nevertheless, Simpson conjectured that each tuple of this type is of geometric origin, see \cite{Simpson92}.

We know that the elements of the monodromy tuples induced by a fourth order differential Calabi-Yau operator lie in $\Sp_4(\mathbb{C})$. By the discussion above, it seems to be promising to investigate those Calabi-Yau operators inducing an $\Sp_4(\mathbb{C})$-rigid monodromy tuple. A bit surprisingly, the classification of all $\Sp_4(\mathbb{C})$-rigid monodromy tuples reveals the following

\begin{ethm}[cf. Theorem \ref{cfg}]
 Each $\Sp_4(\mathbb{C})$-rigid tuple consisting
 of quasi-unipotent elements can be reduced to a tuple of rank one via geometric operations. In particular, it is geometrically constructible using only tuples of rank one and thus of geometric origin.
\end{ethm}

Section three of this article is devoted to the proof of the existence theorem via explicit constructions of those tuples using rational pullbacks, tensor- and Hadamard products of tuples of rank one. A review of all constructions involved, as well as basic facts concerning rigid monodromy tuples, is given in section two. To construct inducing operators of geometric origin, we translate the constructions to the level of differential operators directly rather than choosing an appropriate cyclic vector of the differential system. This is done in section four. The translation of the construction enables us to compute distinguished solutions of the resulting operators explicitly, which is discussed in section five. Finally, we state an explicit construction of those operators whose induced monodromy tuples have a maximally unipotent element in section six. A further investigation yields the following

\begin{conj}
An $\Sp_4(\mathbb{C})$-rigid tuple consisting
 of quasi-unipotent elements and having a maximally unipotent element is induced by a differential Calabi-Yau operator if and only if the elements of its second exterior power lie up to simultaneous conjugation in $\SO_5(\mathbb{Z})$. Furthermore, the inducing operator is unique.
\end{conj}

The construction of differential operators inducing the remaining monodromy tuples will be done in a subsequent article.

We thank Duco van Straten for various fruitfull discussions and suggestions concerning the content of this article.

\section{Rigidity and the middle convolution}

\subsection{Rigidity}

Here we recall the definition of  rigidity in various contexts and 
state  criteria how to read off rigidity via numerical invariants.

\begin{defin}
 \begin{enumerate}
  \item We call $\T$ a \textbf{tuple of rank $n$} if there exist an $r \in \NN$ and $T_i \in \GL_n(\CC), i=1,\ldots, r+1$ such that
         $\T=(T_1,\ldots,T_{r+1})$ and $T_1\cdots T_{r+1}=1$.
         Two tuples are equivalent if they are simultaneously
         conjugate by an element in $\GL_n(\CC).$
  \item  We call a tuple $\T$ \textbf{irreducible}  of rank $n$ 
         if $\T$ generates an irreducible subgroup
         $\langle \T \rangle:=\langle T_1,\ldots,T_{r+1} \rangle$
         of $\GL_n(\CC).$
  \item  We call a tuple $\T$ \textbf{quasi-unipotent} if the 
         eigenvalues of all its elements are roots of unity.
  \item An irreducible tuple $\T$  is called
        \textbf{symplectic}, resp. \textbf{orthogonal}, if 
        $\langle \T \rangle$
        respects a skew-symmetric, resp. a symmetric bilinear form.

\item  Let $G \leq \GL_n(\CC)$ be an irreducible reductive 
algebraic subgroup and
$\langle \T \rangle \leq G$ be irreducible.
We say that $\T$  
is \textbf{$G$-rigid}, if
the following {\it dimension formula} holds:
$$ \sum_{i=1}^{r+1} {\rm codim}(C_G(T_i))=2(\dim(G)-\dim(Z(G)),$$
where $C_G(T_i)$ denotes the 
centralizer of $T_i$ in $G,$
the codimension is taken relative to $G,$ and $Z(G)$ denotes the centre 
of $G.$

 \item  
     We call an irreducible tuple $\T$ of rank $n$
    \textbf{linearly rigid} if $\T$ is $\GL_n(\CC)$-rigid
  and \textbf{symplectically rigid}  if $\T$ is $\Sp_n(\CC)$-rigid.
\end{enumerate}
\end{defin}

The following lemma in \cite{Scott} is often helpful to decide whether a tuple $\T$
is irreducible.

\begin{lemma}[Scott]\label{Scott}
   Let $\T$  act on a vector space $V$. Then 
\begin{eqnarray*}   \sum_{i=1}^{r+1} \rk(T_i-1)&\geq& (\dim(V)-\dim(V^\T))+(\dim(V)-\dim(V^{\check{\T}})),  \end{eqnarray*}
  where $\check{\T}$ denotes the tuple corresponding to dual representation of $\T$ 
  and $V^\T$ the fixed space of $\T$.
   Moreover, if $\T$ is irreducible of rank $n$ then we have
 \begin{eqnarray*} 
   \sum_{i=1}^{r+1} \rk(T_i-1) & \geq & 2n \quad  \mbox{ (Scott formula) and} \\
  \sum_{i=1}^{r+1} \dim (C_{\GL_n(\CC)}(T_i)) &\leq & (r-1)^2 n^2+2 \quad \mbox{(dimension count)}. 
 \end{eqnarray*}
  \end{lemma}

\begin{thm}  \label{symrig}
\begin{enumerate}
\item Let $\T$  be irreducible of rank $n$.
   Then $\T$ is linearly rigid if and only if
 $\T$ is uniquely
 determined by the Jordan forms
 of its elements.
 \item   Let $\T$  be an irreducible symplectic tuple of rank $2m$.
  If there exist only finitely many  tuples $(h_1,\ldots,h_{r+1})$ with $h_1\cdots h_{r+1}=1$ and such that 
$h_i$ is conjugate in $\Sp_{2m}(\CC)$ to $T_i$ then $\T$ is $\Sp_{2m}(\CC)$-rigid, 
 i.e., the  { dimension formula} holds.

 \end{enumerate}
 \end{thm}

\begin{proof}
 \begin{enumerate}
 \item
 The first result goes back to Deligne, Katz and Steenbrink, see e.g. 
 \cite{katz96}.
 \item This statement can be found in \cite{SV}.\qed

 \end{enumerate}

\end{proof}

Alternatively one can consider a tuple
as a finite dimensional
$\CC[F_r]$-module.
For this let  $F_r$ denote  the free group on $r$ generators
$f_1,\ldots,f_{r}$.
Setting $f_{r+1}=(f_1\cdots f_r)^{-1}$
we can view an element
in $\Mod(\CC[F_r])$ as a pair $(\T,V),$ where $V$ is a
vector space over $\CC$ and $\T=(T_1,\ldots,T_{r+1})$ is a tuple in 
$\GL(V)^{r+1}$ such that $f_i$ acts on $V$ via $T_i$ for $ i=1,\ldots,r+1.$
We also assign to $\T$ a tuple $\su=\su_\T=(s_1,\ldots,s_r,s_{r+1}=\infty), $
where $s_1,\ldots,s_r$ are pairwise different elements in $\CC$
with an ordering $s_i<s_j$ in $\su$ if $i<j.$

In a geometric context one can also speak in terms of local systems,
as done in the introduction.

\subsection{Basic properties of the middle convolution}\label{katz}

In this section we recall some of the main properties of the middle convolution
functor $\MC$. 
This functor was introduced by Katz in \cite{katz96} in the category of perverse sheaves. 
A down to earth version
for Fuchsian systems and their monodromy group generators can be found
in \cite{DR07}. 
We recall the main properties of the
convolution that are are stated in  \cite[Section 2]{DR07}.\\

For $(\T,V)\in \Mod(\CC[F_r]),$ where
 $\T=(T_1,\ldots,T_{r+1})\in
\GL(V)^{r+1},$ and $\lambda \in \CC^\times$ one can construct an element
$(C_\lambda(\T),V^r)\in \Mod(\CC[F_r])$ as follows.
For
$k=1,\ldots,r,$ we define $B_k\in \GL(V^r)$ as an element that
maps a vector $(v_1,\ldots,v_r)^{\tr}$
$\in V^r$ to
\[ \left( \begin{array}{ccccccccc}
                  1 & 0 &  & \ldots& & 0\\
                   & \ddots &  & & &\\
                    & & 1 &&&\\
               \lambda (T_1-1) & \ldots& \lambda (T_{k-1}-1)  & \lambda T_{k} & (T_{k+1}-1) & \ldots
&   (T_r-1) \\
     &&&&1&&\\
               &   &  & && \ddots  &   \\
                   0 &  &  & \ldots& &0 & 1
          \end{array} \right)\left(\begin{array}{c}
v_1\\
\vdots\\
\\\vdots\\
\\\vdots\\
\\ v_r\end{array}\right)
.\]
Further we set $B_{r+1}=(B_1\cdots B_{r})^{-1}$.
The subspaces $\K:=\bigoplus_{i=1}^r\K_i,$ where
\[ \K_k = \left( \begin{array}{c}
          0 \\
          \vdots \\
          0 \\
          \ker(T_k-1) \\
            0\\
           \vdots \\
           0
        \end{array} \right)  \quad \mbox{({\it k}-th entry)},\, k=1,\dots,r,\]
and
\[     \L=\bigcap_{k=1}^r \ker (B_k-1)={\rm ker}(B_1\cdots B_r - 1).
\]
of $V^r$ are $\langle B_1,\ldots,B_r \rangle$-invariant.
If $\lambda \not=1$ we have
 $$\L=
\left\langle \left( \begin{array}{c}
            T_2 \cdots T_{r} v \\
                T_3 \cdots T_{r}v \\
               \vdots \\
                        v
            \end{array} \right) \mid v \in \ker(\lambda\cdot T_1\cdots T_r-1) \right\rangle$$
and $$\K+ \L=\K \oplus \L.$$

\begin{defin} Let $(\T,V)\in \Mod(\CC[F_r]).$
\begin{enumerate}

\item We call the $\CC[F_r]$-module $
C_{\lambda}(V):=(C_\lambda(\T),V^r):=((B_1,\ldots,B_{r+1}),V)$
 the \textbf{convolution} of $V$
with $\lambda,$
 where $\su_{C_\lambda(\T)}:=\su_\T.$

\item Let
$\MC_\lambda(\T):=(\tilde{B}_1,\dots,\tilde{B}_{r+1})\in
\GL(V^r/(\K+\L))^{r+1},$ where $\tilde{B}_k$ is induced by the action
of $B_k$ on $V^r/(\K+\L).$ The $K[F_r]$-module
$\MC_{\lambda}(V):=(\MC_\lambda(\T),V^r/(\K+\L))$
 is called the
\textbf{middle convolution} of $\T$ with
$\lambda.$
\end{enumerate}
\end{defin}

\begin{thm} \label{eigen} Let $(\T,V)\in \Mod(\CC[F_r])$ 
 be irreducible. Assume further that if  $\dim(V)=1$
 that at least two of the $T_i,\;i=1,\ldots,r,$ are non trivial.
 Let $\lambda \in \CC^\times.$

\begin{enumerate}

\item If $\lambda \neq 1$ then $${\rm dim}(\MC_\lambda(V))=
\sum_{k=1}^{r} \rk (T_k-1)-
 ({\rm dim}(V)-\rk(\lambda\cdot T_1\ldots T_r-1)).$$

\item If $\lambda_1,\, \lambda_2\in \CC^\times$  then
    \begin{eqnarray*}
\MC_{\lambda_2}\circ \MC_{\lambda_1}(V)\cong \MC_{\lambda_2\lambda_1}(V),&\mbox{where}& \MC_{1}(V)\cong V .
\end{eqnarray*}

\item $\MC_{\lambda}(V)$ is irreducible.
\end{enumerate}
\end{thm}

Obviously, tensoring a linearly rigid tuple with a rank $1$ tuple preserves
linearly rigidity.
Nevertheless this operation plays an essential role in the study of linear rigid
tuples due to Katz' existence algorithm, see Thm.~\ref{linrig}.

\begin{defin} 
  Let $(\T_k,V_k)\in \Mod(\CC[F_{r_i}]),\; k=1,2,$
  be semisimple and $Set(\su)=Set(\su_{\T_1})\cup Set(\su_{\T_2})$, $|Set(\su)|=r+1$,
  where an ordering on $s_i<s_j$ in $\su$ is given by the rule:
  If $s_i,s_j \in Set(\su_{\T_k})$ then $s_i<s_j$ in $Set(\su_{\T_k})$ 
  for $k=1,2$. 
  Thus we consider
  $(\T_1,V_1)$ and $(\T_2,V_2)$ as elements in $ \Mod(\CC[F_{r}]),$
   where $T_{k,j}=1_{V_k}$ if $s_j \not \in Set(\su_{\T_k})$ for $k=1,2.$  
   Then we call
  \begin{eqnarray*}
   \MT(V_1,V_2)&=&V_1\ten V_2,\\
  \MT(\T_1,\T_2)=\MT_{\T_1}(\T_2)&=&(T_ {1,1} \ten T_ {2,1},\ldots,T_ {1,r+1} \ten T_ {2,r+1})
   \end{eqnarray*}
  the \textbf{middle tensor product} of $(\T_1,V_1)$ and $(\T_2,V_2)$.
\end{defin}

\begin{prop}\label{eigen2}
Let $(\T,V)\in \Mod(\CC[F_r])$ 
 be irreducible. Assume further that if  $\dim(V)=1$
 that at least two of the $T_i,\;i=1,\ldots,r,$ are non trivial.
\begin{enumerate}

\item
  If $\T$ is orthogonal,
  symplectic resp.,
   then $\MC_{-1}(\T)$ is
   symplectic,  orthogonal resp.

 \item Let $\T $ be  orthogonal or symplectic and
   $\La_1=(\lambda_1,\lambda_2,(\lambda_{1}\lambda_2)^{-1}),\;
   \La_2=(\lambda_1\lambda_2^{-1},\lambda_{1}^{-1}\lambda_2,1)$
   be  rank $1$ tuples such that $\su_{\La_1}=\su_{\La_2}=(s_i,s_j,s_{r+1})$.
 Then   
$$         \MT_{\La_1^{-1}}
\circ  \MC_{\lambda_1 \lambda_2}\circ  \MT_{\La_2}\circ  \MC_{(\lambda_1 \lambda_2)^{-1}}
   \circ   \MT_{\La_1}(\T)$$
is either orthogonal or symplectic.

\end{enumerate}
\end{prop}  
  
\begin{proof}
For (ii) see \cite[Thm. 5.14]{DR99}.
\qed\end{proof}

 \begin{defin}{\rm 
 Let $\La=(\lambda^{-1},\lambda),\; \su_\La=(0,\infty),$
 be a rank $1$ tuple.
 Then we call
 \begin{eqnarray*}
  \MH_\lambda(\T) & :=& \MC_\lambda(\MT(\T,\La) )
 \end{eqnarray*}
 the \textbf{middle Hadamard} product of $\T$
 with $\lambda$.}
\end{defin}

 The above definition of the middle Hadamard product is motivated by the fact
 that the convolution of $f$ with $x^\mu,\; \lambda=\exp(2\pi i \mu),$ can formally be written as a Hadamard product
 \begin{eqnarray*}
  \int f(x) (y-x)^\mu \frac{dx}{y-x}= \int f(x)x^{\mu} \cdot (\frac{y}{x}-1)^{\mu-1} \frac{dx}{x}.
\end{eqnarray*}

Due to the relation between the convolution
and the Hadamard product we can switch between this both operations freely.

\begin{rem}
 Let $\T$ be  irreducible and $\lambda \in \CC^{\times}.$
 Let $\La=(\lambda,\lambda^{-1}),\; \su_\La=(0,\infty),$
 be a rank $1$ tuple.
 Then 
\begin{eqnarray*}
  \MC_\lambda(\T) & =& \MH_\lambda(\MT(\T,\La) ). 
\end{eqnarray*}
\end{rem}

The middle convolution yields
Katz Existence Theorem, cf. \cite{katz96}.

\begin{thm}\label{linrig}
 Any linearly rigid irreducible tuple $\T$
 of rank $n$
 can be reduced to a rank $1$ tuple via a suitable
 sequence of at most $n-1$ middle convolutions $\MC_{\lambda}$
 and middle tensor products $\MT_\La $ with rank one tuples $\La$. 
\end{thm}

This theorem results in an algorithm to check the existence
of a linearly rigid tuple with given Jordan forms.
Since
$\MC$ is multiplicative and $\La \ten \check{\La}$ is a trivial rank $1$ tuple,
we can invert each step in the algorithm.
 Thus we can  construct a matrix representation of $\T$.

\begin{ex}\label{Hyp}
  The tuple 
  \[ \T=(T_0,T_1,T_\infty):=\MH_{\beta} \circ \MH_{\beta^{-1}}  \circ \MH_{\alpha}\; ( 1,\alpha,\alpha^{-1}),\quad \alpha, \beta \in \CC^{\ast}\setminus \{1\}\]
 is a
  symplectic  tuple of rank $4$.
  Using the methods described in this section we can compute $\T$  explicitly.
 Setting 
$ A=\alpha +\alpha^{-1}-2,\;  B=\beta +\beta^{-1}-2$
  we get
  \begin{eqnarray*} T_0=\left(\begin{array}{cccc}
    1 & 1 &  0 &  0 \\ 
      0 &  1 &  -1 &  0 \\ 
      0 &  0 &  1 &  1 \\
      0 &  0 &  0 &  1    
      \end{array}\right),&&
  T_1=    \left(\begin{array}{cccc}
    1 & 0 &  0 &  0 \\ 
      0 &  1 &  0 &  0 \\ 
      0 &  0 &  1 &  0 \\
      AB  &AB&A+B &1   
     \end{array}\right).
 \end{eqnarray*}
  This is a special case of a monodromy tuple of a generalized hypergeometric differential equation.
  Those monodromy tuples were first described by Levelt \cite{Levelt}.
  A detailed study of the monodromy we refer to  the paper \cite{BH90} of  Beukers and Heckman.
\end{ex}

\subsection{The numerology of the middle convolution}
We recall the effect of the middle convolution 
on the Jordan  forms of the 
local monodromy, given by Katz in \cite{katz96}, Chap. 6:

For $i= 1,\ldots,r+1,$ we write 
 ${\J}(T_i)=\oplus_{\rho \in \CC } \; \rho \J(j)^{v(i,\rho,j)},\;v(i,\rho,j) 
 \in \NN_0,$ 
 as a direct sum of Jordan blocks $\rho \J(j)$ of size $j$ with respect to the eigenvalue $\rho$ with multiplicity $v(i,\rho,j)$.
We also write $T_0$ (resp. $T_\infty $) for the monodromy at $0$ (resp. $\infty$).

\begin{prop}\label{propF} Let $ \T$ be irreducible of  rank $n$ and
 $ \lambda \neq 1$.
The transformation of the Jordan forms of its elements
under the middle convolution
is given by
\begin{eqnarray*}
 {\J}(\MC_\la(T_i))&=&
    \bigoplus_{\substack{\rho \in \CC \setminus\{1,\la^{-1}\}}}
         \la \rho \J(j)^{v(i,\rho,j)}
    \bigoplus_{\substack{j\geq 2}} \la \J(j-1)^{v(i,1,j)}
    \\
  \;(i =1,\ldots, r) &&  \quad \; \bigoplus \J(j+1)^{v(i,\la^{-1},j)}\;
     \bigoplus \J(1)^{k_i}        \\ 
  \J(\MC_\la(T_{r+1}))&=&
    \bigoplus_{\rho \in \CC\setminus\{1,\la\}} 
    \la^{-1} \rho \J(j)^{v(r+1,\rho,j)}
    \bigoplus \J(j-1)^{v(r+1,\la,j)} \\
     &&
    \quad \; \bigoplus \la^{-1} \J(j+1)^{v(r+1,1,j)}
  \bigoplus \la^{-1} \J(1)^{k_{r+1}},
\end{eqnarray*}
where $k_j$ is determined by 
\begin{eqnarray*}\rk(\MC_\lambda(\T))&=&\sum_{i=1}^r \rk (T_i-1)+\rk(\lambda^{-1} T_\infty -1)-n.
\end{eqnarray*}

\end{prop}

This also shows that
the middle convolution $MC_\lambda$ preserves linear rigidity
by Thm.~\ref{symrig}.

From the definition of the middle Hadamard product and the above
proposition we can derive  the Jordan forms of $\MH_\lambda(\T)$:

\begin{prop}\label{propH} Let $ \T$ be irreducible of  rank $n$ and
 $\lambda \neq 1.$
 The transformation of the Jordan forms of its elements
 under the middle Hadamard product is given by
\begin{eqnarray*}
   {\J}(\MH_\la(T_i))  &=&
 \bigoplus_{{ \rho \in \CC \setminus\{1,\la^{-1}\}}} \;
    \la \rho \J(j)^{v(i,\rho,j)} 
 \bigoplus \;\J(j+1)^{v(i,\la^{-1},j)} \\
  (i\neq 0,r+1) &&   \quad \;\;  \bigoplus_{j\geq 2} \;\; 
        \la \J(j-1)^{v(i,1,j)}
  \bigoplus \;\J(1)^{k_i} \\
 \J(\MH_\la(T_{0}))&=&
 \bigoplus_{{ \rho \in \CC \setminus\{1,\la^{-1}\}}} \;
    \la \rho \J(j)^{v(i,\rho,j)}
 \bigoplus \;\J(j+1)^{v(i,\la^{-1},j)} \\
    &&  \quad \;\;  \bigoplus_{j\geq 2} \; \;\;
        \la \J(j-1)^{v(i,1,j)}
  \bigoplus \;\J(1)^{k_i} \\
\J(\MH_\lambda(T_{r+1}))& =& 
  \bigoplus_{\rho \in \CC\setminus\{1,\lambda^{-1}\}}\; 
                   \rho \J(j)^{v(r+1,\rho,j)}  \quad 
        \bigoplus_{j\geq 2} \;
                  \J(j-1)^{v(r+1,1,j)} \\
 && \quad \; \;\bigoplus \; \lambda^{-1}  \J(j+1)^{v(r+1,\lambda^{-1},j)} 
                 \bigoplus \;\lambda^{-1}\J(1)^{k_{r+1}}
\end{eqnarray*}

where $k_j$ is determined by 
\begin{eqnarray*}\rk(\MH_\lambda(\T))&=&\sum_{T_{i}\neq T_{0} } \rk (T_i-1)+\rk(\lambda^{-1} T_{0} -1)-n.
\end{eqnarray*}

\end{prop}

\section{Classification of symplectically rigid tuples of rank four}

This section is devoted to the classification of
symplectically rigid tuples of rank four.
In particular we show.

\begin{thm}\label{cfg}
 Let $\T$ be a symplectically rigid tuple of rank four consisting
 of quasi-unipotent elements. 
  Then $\T$ is coming from geometry. i.e
  $\T$ is a monodromy tuple of a factor of a Picard-Fuchs equation.
  Moreover it can be constructed by a sequence
  of geometric operations starting with a rank one tuple.
  These geometric operations include  tensor products, rational pullbacks 
  and the middle convolution.
\end{thm}

Roughly speaking the proof of Thm.~\ref{cfg}
is based on the following steps:\\

STEP one: Using Thm.~\ref{symrig} (ii)
   we classify in Table~2 all possible symplectically rigid irreducible tuples $\T$
   of rank $4$ via the tuples 
   \[P_i:=(\dim C_{\Sp_4(\CC)}(T_1),\ldots,\dim C_{\Sp_4(\CC)}(T_{r+1}))\]
    of the centralizer dimensions of their elements.
     We list these centralizer dimensions in
     Table~1.
   Via M\"obius transformations, which are
   sharply 3-transitive, and more generally
   the action of the Artin braid group $\B_r$ on  $\T$
    that permutes the local monodromies, we can order the entries according to
   increasing dimensions.
   Thus we get the finite list $P_1,\ldots, P_5$
   in Table ~2.
   Further, we refine these cases by the subcases   
    \[ P_i(\dim C_{\GL_4(\CC)}(T_1),\ldots,\dim C_{\GL_4(\CC)}(T_{r+1})).\]
   E.g.,
   the $P_3(4,8,10,10)$ case denotes  irreducible quadruples $\T$ with  
   \begin{eqnarray*}
   (\dim C_{\Sp_4(\CC)}(T_1),\ldots,\dim C_{\Sp_4(\CC)}(T_{4}))&=&(2,6,6,6)
   \end{eqnarray*}
   and
    \begin{eqnarray*} 
 (\dim C_{\GL_4(\CC)}(T_1),\ldots,\dim C_{\GL_4(\CC)}(T_{4}))&=&(4,8,10,10).\end{eqnarray*}
   Moreover Table~1  shows that 
$\J(T_1)\in \{ \pm \J(4), (-\J(2),\J(2)),(x\J(2),x^{-1}\J(2)),(x,y,y^{-1},x^{-1})\},$
   $\J(T_2)=(-1,-1,1,1)$ and $ \J(T_3)=\J(T_4)=(\J(2),1,1)$.

 STEP two: 
          The irreducibility condition  restricts the possible tuples of
          Jordan forms via the Scott formula or the dimension
          count in Lemma~\ref{Scott}.
          E.g. there is no rigid tuple of with Jordan forms
          $(\J(4),\J(4),(\J(2),1,1))$
          in the $P_1(4,4,10)$ case, as $7=\sum_i \rk(T_i-1)<2\cdot 4$.

 STEP three:  We check whether $\T$ is linearly rigid using the dimension count
 by Thm.~\ref{symrig} (i).
 In the positive case the claim follows from Katz' algorithm, see Thm.~\ref{linrig}.
 Moreover the algorithm 
 imposes  the conditions for the
 existence of such a $\T$ depending on the eigenvalues of the $T_i$.

 STEP four:  Using the operations in
 Prop.~\ref{eigen2}  we try to construct a tuple $\tilde{\T}$ in 
 an orthogonal group of dimension $3$, $4$, $5$ or $6$.
 Due to the exceptional isomorphisms we have
 \[ \Sym^2 \Sp_2(\CC) = \SO_3(\CC),\quad \Sp_2(\CC) \ten \Sp_2(\CC) =\SO_4(\CC), \]
 \[ \Lambda^2 \Sp_4(\CC) = \SO_5(\CC),\quad \Lambda^2 \SL_4(\CC) = \SO_6(\CC), \]
 which can again result  in linearly
 rigid tuples.
 E.g. an orthogonal triple $\T$ of rank $3$ with $\J(\T)=(\J(3),\J(3),\J(3))$ yields
 a linearly rigid triple $\tilde{\T}$ of rank $2$ with $\J(\tilde{\T})=(\J(2),\J(2),-\J(2)).$\\

It turns out that in all  $P_i$ cases we either get  contradictions
to the irreducibility or we end up with a rank one tuple.
In the latter case we obtain 
a suitable
sequence of operations
that allows us to construct this symplectically
rigid tuple $\T$  of rank four, since each operation is invertible.
Moreover if the symplectically
rigid tuple of rank four is quasi-unipotent it turns out that it can be constructed using only geometric operations cf. \cite[Chap. II]{Andre}.\\

We begin with  Step one and classify the Jordan forms
in $\Sp_4(\CC)$ and their centralizer dimensions.
Since  $\Lambda^2 \Sp_4(\CC)=\SO_5(\CC)$
we also determine the Jordan forms in $\SO_5(\CC).$

{{\small

\begin{equation*}\label{table0}
\begin{array}{|c|c|c|c|}
\hline
 \textrm{Jordan form in }  & \textrm{ Jordan form in}&\mbox{\rm centralizer dimension in} &  
\textrm{conditions}\\
  \Sp_4(\CC)             &\SO_5(\CC) & \Sp_4(\CC)  \quad\quad \GL_4(\CC)&  \\
\hline 
 \pm (1,1,1,1)&(1,1,1,1,1) &10 \quad\quad\quad  16&\\
\pm(\Jordan(2),1,1)& (\Jordan(2),\J(2),1)& 6\quad\quad\quad 10&\\
 \pm(\Jordan(2),\Jordan(2))&(\J(3),1,1)&4\quad\quad\quad 8&\\
\pm \Jordan(4) & \J(5)& 2\quad\quad\quad 4&\\
\hline
 (-1,-1,1,1)& (-1,-1,-1,-1,1)&6\quad\quad\quad  8 & \\
\pm (-\Jordan(2),1,1)& (-\J(2),-\J(2),1)& 4\quad\quad\quad 6&\\
 (-\Jordan(2),\J(2))&(-\J(3),-1,1)&2\quad\quad\quad 4&\\
\hline
(x,x,x^{-1},x^{-1})& (x^2,1,1,1,x^{-2})&4 \quad\quad\quad 8& x^2 \neq 1 \\
(x,1,1,x^{-1})&(x,x,1,x^{-1},x^{-1})&4\quad\quad\quad 6& x^2 \neq 1  \\
(x\Jordan(2),x^{-1}\Jordan(2))&(\J(3),x^2,x^{-2})& 2\quad\quad\quad 4 & x^2 \neq 1 \\
(x ,x^{-1},\Jordan(2))&(x\J(2),x^{-1}\J(2),1)&2 \quad\quad\quad 4& x^2 \neq 1\\
\hline
(x,y,y^{-1},x^{-1})&(xy,xy^{-1},1,x^{-1}y,x^{-1}y^{-1})&2\quad\quad\quad 4 &x^2,y^2 \neq 1 \\
&& & x\neq y^{\pm 1}\\
\hline
\end{array}
\end{equation*}}
\nopagebreak
\begin{center} {\bf Table~1:} The Jordan forms of  elements in $\Sp_4(\CC)$ and $\SO_5(\CC) =\Lambda^2\Sp_4(\CC)$.
\end{center}}

\begin{center}
\begin{tabular}{|c|c|c|c|} 
\hline 

{\rm case}& $ {\rm subcases}$& \mbox{ remarks}\\
\hline
$P_1$&  (4,4,10) & \mbox{ lin. rigid} \\
 (2,2,6)    &  (4,4,8) & $\Lambda^2$ \mbox{ lin. rigid} \\
\hline 
$P_2$&   (4,6,6) &    \\
(2,4,4)     &   (4,6,8) & \mbox{ lin. rigid} \   \\
     &   (4,8,8) &\mbox{ red. (dimension count)} \    \\
\hline 
$P_3$& (4,10,10,10) & \mbox{red. (Scott)} \\
(2,6,6,6)     & (4,8,10,10) & \mbox{} \\ 
     & (4,8,8,10) & \mbox{$\Lambda^2$ red. } \\ 
     & (4,8,8,8) & \mbox{$\Lambda^2$ red. } \\ 
\hline 
$P_4$&(8,8,10,10) & \mbox{red. (dimension count) } \\
 (4,4,6,6)    & (6,8,10,10) & \mbox{lin. rigid} \\ 
     & (6,6,10,10) & \\
     & (8,8,8,10) & \mbox{lin. rigid} \\
     & (6,8,8,10) &  \\ 
     & (6,6,8,10) & \mbox{ $\Lambda^2$ red.} \\
     & (8,8,8,8) & \mbox{$\Lambda^2$ red.} \\
     & (6,8,8,8) & \mbox{$\Lambda^2$ red. } \\ 
     & (6,6,8,8) & \mbox{$\Lambda^2$  lin. rig.}\\  
\hline 
$P_5$&(10,10,10,10,10) & \mbox{lin. rig.} \\
 (6,6,6,6,6)    & (8,10,10,10,10) & \mbox{red. (Scott)} \\ 
     & (8,8,10,10,10) & \mbox{red. (Scott)} \\ 
     &  (8,8,8,10,10) & $\Lambda^2$ \mbox{lin. rig.} \\ 
     & (8,8,8,8,10) & $\Lambda^2$ \mbox{red.} \\ 
     & (8,8,8,8,8) & $\Lambda^2$ \mbox{red.} \\ 
\hline 
\end{tabular}

\vspace{.5cm}
{{\bf Table~2:} {The centralizer conditions for symplectically rigid tuples}}
\end{center}

 In the following sections we  rearrange the order of the centralizer
 dimensions in Table~2 via M\"obius transformations to simplify the proofs.
 If $\T$ is a triple we can assume that 
 $\su_\T=\{0,1,\infty\}$.
 Thus we also index $\T=(T_0,T_1,T_\infty).$
 E.g., a linearly rigid tuple in the $P_1(4,10,4)$ case such that
      $T_0$ is unipotent, can be written as a sequence of 3 Hadamard products starting from a rank $1$ tuple, see
 Ex.~\ref{Hyp}. 
 However in the $P_1(4,4,10)$ case the Katz algorithm requires 
 additional tensor products with rank $1$ tuples.

  To abbreviate the notations we denote
  by $\J(\T)$ the tuple of Jordan forms.
  Further we write  
  $\J_s(\T)$ for 
 $(\J_s(T_1),\ldots,\J_s(T_{r+1}))$, where
  $\J_s(T_i)$  denotes  the semisimple part of $\J(T_i)$.

\subsection{The $P_1$ case}

\subsubsection{The $P_1(4,10,4)$  case}
\begin{rem}
 We omit the  linearly rigid $P_1(4,10,4)$ case.
 This well studied 
 case corresponds to  monodromy tuples of  generalized hypergeometric
 differentials equation of order $4$ and is settled by Katz' algorithm.
 For an example, where $T_0$ is maximally unipotent,  see Ex.~\ref{Hyp}.
\end{rem}

\subsubsection{The $P_1(4,8,4)$  case}
\begin{thm}\label{SymP1}
  A symplectically
  rigid  tuple $\T$ in the case $P_1(4,8,4)$ 
 can be obtained from a rank one tuple
 using the middle Hadamard product and tensor products.
 Moreover  the tuple $\T$ can be written  
  \[ \T=\MH_{-1}(\Lambda^2(\S )), \]
   where $\S$ is linearly rigid rank $4$  triple 
   containing a transvection.
 \end{thm}

\begin{proof}
   By Thm.~\ref{eigen} and  Cor~\ref{eigen2}  the Hadamard product 
   $ \MH_{-1}(\T)$
   yields an irreducible 
   orthogonal triple of     rank $m,$ where 
    \[ m=\rk(-T_0-1)+ \rk(T_1-1)+\rk(T_{\infty}-1) -4 \in \{4,5,6 \}.\]
   Hence we can apply one of the identities
   \begin{eqnarray*}
    \Lambda^2 \Sp_4(\CC)  = \SO_5(\CC) ,&&  \Lambda^2 \SL_4(\CC)  = \SO_6(\CC) 
   \end{eqnarray*}
    to obtain a triple of rank $4$ containing a transvection, since by Prop.~\ref{propH}
     \begin{eqnarray*} \J(\MH_{-1}(T_1)) & =&(\J(2)^2,\J(1)^{m-4}).\end{eqnarray*}
    For $m=4$ we use the natural embedding of $\GO_4(\CC)$ in $\SO_5(\CC).$ 
    Thus the triple is linearly rigid  the claim follows from Katz' algorithm. 
\qed\end{proof}

\begin{rem}
 The construction of $\T$ is in general not unique.
In the above case one could also get $\T$ by using
 that $\Lambda^2(\T)$ yields a linearly rigid tuple and
 then apply Katz' algorithm. However in this construction
 the computation of the matrix representation of $\T$ is more
 complicated.  
\end{rem}

\begin{cor}\label{P1sym}
 Let $\T$ be as in   Thm.~\ref{SymP1}
 such that  $T_0$ is maximally unipotent and
  $\J_s(T_\infty)= (xy,xy^{-1},x^{-1}y,(xy)^{-1})$.
Then
\[ T_0=\left(\begin{array}{cccc}
    1 &  ab &  0 &  (a+b)^2 \\ 
      0 &  1 &  1 &  0 \\ 
      0 &  0 &  1 &  -ab \\
      0 &  0 &  0 &  1    
   \end{array}\right), \quad
 T_1=\left(\begin{array}{cccc}
      -1 &  -2ab &  0 &  0 \\ 
      0 &  1 &  0 &  0 \\ 
      1 &  ab &  1 &  0 \\ 
      0 &  -1 &  0 &  -1  
    \end{array}\right),\]
where $ a=x+\frac{1}{x},\quad  b=y+\frac{1}{y},\quad x,y \in \CC^{\ast}$
and $ab \neq 0$.
The tuple $\T$ can be obtained as follows.
\begin{eqnarray*} \T&=&\MH_{-1}\circ \MT_{\La_1}(\Lambda^2 \S),\mbox{ where}\\
 \S&=&\MH_{(ix)} \circ \MH_{-(ix)^{-1}}  \circ \MT_{\La_1} \circ \MH_{-iy}(\La_0)
\end{eqnarray*}
with $\La_0=(1,(iy)^{-1},iy)$ and $\La_1=(-1,1,-1)$ are  a  rank $1$ triples.
Further,
 $\MT_{(i,1,i^{-1})}\S$ is symplectic and linearly rigid of rank $4$ with
\[(i\J(S_0),\J(S_1),-i\J_s(S_\infty))=((i\J(2),-i\J(2)),(\J(2),1,1),(x,y,y^{-1},x^{-1})).\] 
\end{cor}

\begin{proof}
 The tuple $\T$ can be constructed 
 using the matrices in Section~\ref{katz} according
 to the given sequence of Hadamard products and tensor products. 
 Prop.~\ref{propH} allows to keep track of the change of Jordan forms under Hadamard product.
 We demonstrate this for the case,
 where $x,x^{-1},y,y^{-1}$ are pairwise different:
  We start with a rank $1$ triple $ \La_0=(1,(iy)^{-1},iy)$ and apply $\MH_{-iy}$.
  This yields a rank $2$ triple with Jordan forms
  $(\J(2), (-1,1), (-iy^{-1},-iy))$. Then we proceed with the tensor product
  $\MT_{\La_1}$ and so on. 
  Tabulating the operations and the change of the Jordan forms we get
 \[  \begin{array}{c||c||c|c|c|}
           \rk & \mbox{operation} &  \mbox{Jordan} &  \mbox{forms}  &  \\ \hline
       1 &   & (1) &(iy) &((iy)^{-1}) \\ \hline
       2 & \MH_{-iy} & \J(2) & (-1,1)& (-iy^{-1},-iy) \\\hline
       2  &  \MT_{\La_1}&  -\J(2) & (-1,1)&  (iy^{-1},iy) \\\hline
       3 &    \MH_{-(ix)^{-1}}&  (-\J(2),1) & ((ix)^{-1},1,1) &  (iy^{-1},iy,ix^{-1}) \\ \hline
      4 &    \MH_{(ix)}&  (-\J(2),\J(2)) & ((\J(2),1,1) &  (iy^{-1},iy,ix^{-1},ix) \\\hline
     5 &  \Lambda^2 & (-\J(3),1,1) & (\J(2),\J(2),1) &  -(xy^{-1},xy,1,(xy)^{-1},x^{-1}y) \\\hline
    5 &  \MT_{\La_1}&   (\J(3),-1,-1) & (\J(2),\J(2),1) &  (xy^{-1},xy,1,x^{-1}y^{-1},x^{-1}y) \\\hline
   4&\MH_{-1}&  \J(4) & (-1,-1,1,1) &  (xy^{-1},xy,x^{-1}y^{-1},x^{-1}y) \\
 \end{array}.\]
 By Prop.~\ref{eigen2} iii) we know that $\MT_{(i,1,i^{-1})}(\S)$ is symplectic and
  we use that $\Lambda^2 \Sp_4(\CC) = \SO_5(\CC)$.
  In the general case the Jordan form of the third element (in each step)
  is obtained by
  replacing $k$ equal eigenvalues $z$ by $z\J(k)$.

  The conditions for the irreducibility follow from the fact that
 the middle Hadamard product  has to be non trivial in each step,
 i.e. $i\neq \pm x, \pm y$ by Thm.~\ref{eigen}.
 Thus $ab\neq 0$.
\qed\end{proof}

\begin{cor}\label{bspZ}
  Let $\T$ be as in  in Cor.~\ref{P1sym}.
  Then the Zariski closure of $\langle \T\rangle $ is $\Sp_4(\CC).$
  Moreover if $ab, a^2+b^2\in \ZZ$ then $\langle \T\rangle $ 
  is contained up to conjugation
  in $\Sp_4(\ZZ)$.
  Further, if $\T$ is quasi-unipotent
   then the conditions are also necessary.
\end{cor}

\begin{proof}
  Since $\J(T_1)=(-1,-1,1,1)$ 
  the Zariski closure of $\langle \T\rangle $ is not $\Sym^3(\SL_2(\CC))$
 and the first statement follows from
 Cor.~\ref{Zar}.
 The  matrix representation shows that the conditions are sufficient.
 The necessary condition for the group $\langle \T\rangle$ to be 
 contained 
 in $\Sp_4(\ZZ)$ is that all traces of all elements are integers.
 Hence 
 \[\tr(T_\infty)=ab, \quad \tr(T_\infty^2)=(a^2-2)(b^2-2)=(ab)^2+4-2(a^2+b^2) \in \ZZ.\]
 Hence $ab,2(a^2+b^2)\in \ZZ$.
 But if  $a,b$ are sums of roots of unity then $2(a^2+b^2)\in \ZZ$ implies
 $(a^2+b^2)\in \ZZ$.
\qed\end{proof}

\subsection{The $P_2$  case}

\subsubsection{The $P_2(4,6,6)$  case}
\begin{thm}\label{P2466}
  Let $\T$ be a symplectically
  rigid  tuple 
  in the case $P_2(4,6,6),$
  where 
\begin{eqnarray*} 
\J_s(\T)=((z_1z_2,z_1z_2^{-1},z_1^{-1}z_2,(z_1z_2)^{-1}), &(1,-x^2,-x^{-2},1),& (y^2,-1,-1,y^{-2})),
 \end{eqnarray*}
with $ x,y,z_1,z_2 \in \CC^{\ast}$.
 Then $\T$ can be written 
  \begin{eqnarray*}  \T=\MH_{-1}(\MT(\S_1,\S_2)),&\mbox{where }&
                     \S_i=\MT_{\La_{2i}} (\MH_{z_ixy^{-1}}  \La_{1i})
 \end{eqnarray*}
 with
 $\La_{2i}=(z_i^{-1},x^{-1},z_ix),\; 
 \La_{1i}=(z_i^2,z_i^{-1}xy,(z_ixy)^{-1})$,  $i=1,2$.
 \end{thm}

\begin{proof}
 The tuple
 \[ \S= \MT_\La \circ \MC_{-1}(\T),\quad \La=(-1,1,-1), \]
 is an orthogonal triple of rank 
\[ m=\rk(T_0-1)+\rk(T_1-1)+\rk(-T_\infty-1)-4 \in \{3,4\} \]
 by Thm~\ref{eigen} and Prop.~\ref{eigen2} (ii).
 Using that
 \begin{eqnarray*} \SO_4(\CC) = \Sp_2(\CC) \ten \Sp_2(\CC),&&\SO_3(\CC)= \Sym^2\Sp_2(\CC)
   \end{eqnarray*} 
 we can write $\S$ as $\S=\S_1\ten \S_2$ with
 \begin{eqnarray*} (\J(S_{i0}),\J(S_{i1}),\J(S_{i\infty}))&=&((z_i,z_i^{-1}),(x,x^{-1}),\pm (y,y^{-1})),\;i=1,2.
   \end{eqnarray*}  
 Since $\S_1$ and $\S_2$ are linearly rigid the claim follows from Katz' algorithm.
\qed\end{proof}

\begin{cor}\label{symP2}
 Let $\T$ be  as in Thm.~\ref{P2466},
 such that  $T_0$ is maximally unipotent.
 Then
\[T_0= \left(\begin{array}{cccc}
1& -a+b& a& -2 \\ 
0& 1&      -2&  b\\  
0& 0&      1&   a-b\\
0& 0&      0&   1     
\end{array}\right),\quad
 T_1=\left(\begin{array}{cccc}
1& 0&     0&   0 \\     
0& 1&     0&   0\\      
0& 2&     1&   -a \\  
2& a+b& a& -a^2+1
\end{array}\right),\quad \]
where
$a=x+\frac{1}{x},\; b=y+\frac{1}{y}$ and $a\neq b$.

The tuple $\T$ can be written as
\[ \T=\MH_{-1}(\Sym^2 \S),\quad  \S=\MT_{\La} \circ \MH_{{xy^{-1}}} (\La_0),\]
where $\La=(1,x^{-1},x)$ and $\La_0=(1,xy,(xy)^{-1})$ are rank $1$ triples with
 $\su_\La=\su_{\La_0}=(0,1,\infty)$.
\end{cor}

\begin{proof}
 The proof is analogous to the proof of Cor~\ref{P1sym}.
\qed\end{proof}

\begin{cor}
  Let $\T$ be  as in Cor~\ref{symP2}.
 Then  the Zariski closure of $\langle \T \rangle$
  is $\Sp_4(\CC)$ if and only if
 $a^2\neq 1$ and $b^2 \neq 1$.
 The generated group
 is up to conjugation contained 
 in $\Sp_4(\ZZ)$ if and only if 
 $a^2,b^2,ab\in \ZZ$.
\end{cor}

\begin{proof}
 By construction there are at most two symplectically rigid
 tuples with given Jordan forms since $\Sym^2$ does not act  bijectively
 on the Jordan forms.
  However if  $a=-b$
 then the Jordan forms determine
 the tuple $\T$ uniquely since a rank $2$ triple with Jordan forms
 $(\J(2),(x,x^{-1}),(x,x^{-1}))$ is reducible.

 Further if $a^2=b^2=1$ then $x,y$ are sixth roots of unity
 and $\T$ can be also written as $\Sym^3$ of a rank $2$ tuple.
 By uniqueness and  Cor.~\ref{Zar} the first claim follows.
 
 If the generated group is up to conjugation contained 
 in $\Sp_4(\ZZ)$ then
 the trace condition implies $a^2,b^2\in \ZZ$.
 By construction  
 the middle convolution $\MC_{-1}$ and taking $\Sym^2$ are compatible with
 the action of a field automorphism.
 Thus if  $ab \not \in \ZZ$ then there exists a
 $\sigma \in \Gal(\QQ(a,b)/\QQ)$
 such that $\si(a)=a$ and $\sigma(b)=-b$.
 But then we get $\T^{\si}= \T$ and $\S^{\si} \neq \S,$ a contradiction.
 The matrix representation shows that these conditions are also sufficient.
 Namely, if $a,b \not \in \ZZ$, but $ab \in \ZZ$  then $a=n_1 \sqrt{d}$
 and $b=n_2 \sqrt{d}$. 
 Thus if we conjugate the matrices in Cor~\ref{symP2} 
 by $\diag(\sqrt{d},1,1,\sqrt{d})$
 we get a representation in $\Sp_4(\ZZ).$
\qed\end{proof}\\

\subsubsection{The $P_2(4,6,8)$  case}
Since the proofs of the statements in the linearly rigid $P_2(4,6,8)$ case
are analogous to the proofs before we omit them.

\begin{thm}\label{linrigP2}
 A linearly rigid  tuple $\T$ in the case $P_2(4,6,8)$, where 
\[\J_s(\T)=((z_1,z_2,z_2^{-1},z_1^{-1}),(1,1,y,y^{-1}), (x,x,x^{-1},x^{-1})),\] 
 can be obtained 
 as
\[ \T= \MT_{\La_3}\circ\MH_{xz_1} \circ \MT_{\La_2} \circ \MH_{(xz_1)^{-1}}\circ \MT_{\La_1} \circ \MH_{yz_1z_2} (\La_0),\]
where $\La_3=(z_1^{-1},1,z_1),\; \La_2=(z_1^2,1,z_1^{-2}),\;
\La_1=((z_1z_2)^{-1},y^{-1},yz_1z_2)$ and\\ $\La_0=(z_2^2,y(z_1z_2)^{-1},z_1z_2^{-1}y^{-1})$.
\end{thm}

\begin{cor}\label{linrigP2MUM}
 Let $\T$ be as in Thm.~\ref{linrigP2}
 such that  $T_0$ is maximally unipotent.
Then
\[ T_0=\left(\begin{array}{cccc}
1&-1&0&            a-2 \\
0&1& a-2&0 \\
0&0& 1&            -b+2 \\
0&0& 0&            1
\end{array}\right),\quad
 T_1=
\left(\begin{array}{cccc}
1&0&0&0 \\
0&1&0&0 \\
0&1&1& b-2 \\
1&0&1& b-1 \\
\end{array}\right),\]
where $a=x+\frac{1}{x}, \quad  b=y+\frac{1}{y},\quad x,y \in \CC^{\ast}\setminus \{1\}.$
The tuple $\T$ can be obtained via
\[ \T=\MH_{x}  \circ \MH_{x^{-1}}  \circ \MT_{\La_1}\circ \MH_y (\La_0),\]
where $\La_1=(1,y^{-1},y) $ and $\La_0=(1,y,y^{-1})$ are rank $1$ triples.
\end{cor}

\begin{cor}
 Let $\T$ be as in  Cor.~\ref{linrigP2MUM}.
 Then $\langle \T \rangle $ 
is contained up to conjugation
 in $\Sp_4(\ZZ)$ if and only if 
 $a,b\in \ZZ$.
 The  Zariski closure of $\langle \T \rangle $ is $\Sp_4(\CC)$
 if and only if $a\neq 0$ and $b\neq -1$.
 \end{cor}

\subsection{The $P_3,\;P_4$ and $P_5$ cases}

In this section we show that in the cases $P_3,\;P_4$ and $P_5$  all symplectically rigid
tuples $\T$
can be reduced  via geometric operations to rank $1$ tuples.
Since we prefer to work with the convolution
we index $\T=(T_1,\ldots,T_r,T_{r+1}=T_\infty)$.
In order to shortcut the following proofs we use without citing
  that the application of $\MC_{-1}$ changes a symplectical
tuple into an orthogonal one by Prop.~\ref{eigen2} (ii) whose rank is given
by Thm.~\ref{eigen}.
Moreover, due to Katz' algorithm it suffices to relate $\T$ to a linearly rigid tuple.

\subsubsection{The $P_3$   case}

\begin{thm}\label{P3}
 In all the $P_3$ cases
 a symplectically rigid  tuple $\T$ can be reduced via middle convolution operations,
taking  tensor products and rational pullbacks to a rank $1$ tuple.
 Further there exists no $\T$ with a maximally unipotent element.
\end{thm}

\begin{proof}
\begin{enumerate}
 \item The case $P_3(4,10,10,10)$ is ruled out by the Scott formula.
 
\item In the case $P_3(4,8,10,10)$ 
      the Scott formula implies that $\rk(T_1-1)=\rk(T_1+1)=4.$
      Let $\La_1=(\la,1,1,\la^{-1})$ such that $\rk(T_1\la-1)=3.$
      Then
      \begin{eqnarray*}      \T_1&=&\MC_{\la^{-1}}\MT_{\La_1}(\T) \end{eqnarray*}
      is a rank $3$ tuple.
      Taking
      $\La_2=(\la^{-1},-\la,1,-1)$ and  $\La_3=(-1,\la^{-1},1,-\la)$ 
       we obtain a rank $2$ quadruple 
      \begin{eqnarray*}  \S&=& \MT_{\La_3} \circ   \MC_{-\la} \circ \MT_{\La_2}(\T_1) \end{eqnarray*}
      in $\GO_2(\CC)$ by  Prop.~\ref{eigen2} (iii). 
      If $\T$ is quasi-unipotent the generated group
      is finite and therefore a pullback of  a linearly rigid monodromy tuple of
      a Gauss hypergeometric differential equation by a well known
      result of Klein (cf. \cite[Thm. 3.4]{BalDwo79}).
      In any case a quadratic pullback yields 
      a direct sum of two rank $1$ tuples.
\item Taking $\Lambda^2$ in the case $P_3(4,8,8,10)$ 
      we obtain a reducible tuple in $\SO_5(\CC)$ by the Scott formula.
      This excludes $\J(T_1)=\J(4)$ by
      Cor.~\ref{Zar}.
      Let $\La_1=(\la,1,1,\la^{-1})$ such that $\rk(T_1\la-1)=3.$
      Then
      \begin{eqnarray*}      \T_1&=&\MC_{\la^{-1}}\MT_{\La_1}(\T) \end{eqnarray*}
      is a rank $4$ tuple.
      Taking
      $\La_2=(\la^{-1},-\la,1,-1)$ and  $\La_3=(-1,\la^{-1},1,-\la)$ 
       we obtain a rank $4$ quadruple 
      \begin{eqnarray*}  \S&=& \MT_{\La_3} \circ   \MC_{-\la} \circ \MT_{\La_2}(\T_1) \end{eqnarray*}
      in $\GO_4(\CC)$ by Prop.~\ref{eigen2} (iii).
      A quadratic pullback yields a $5$-tuple
      $\T_2$ with
      Jordan forms 
      \[ ((\J(2),\J(2)), (\J(2),\J(2)),(\la,\la,\la^{-1},\la^{-1}),(\la,\la,\la^{-1},\la^{-1}),(\la_2^2,1,1,\la_2^{-2})), \]
      where $\rk(S_1-\la_2)=3.$       
      Hence 
      $\T_2$ can be written as a tensor product of two 
      $5$-tuples 
      $\S_1$ and $\S_2$ of rank $2$ having two trivial entries.
      Since the $\S_i$ are linearly rigid the claim follows.
\item We can exclude the case $P_3(4,8,8,8)$.
      Since 
      $ \MC_{-1}(\T)$ 
      yields an orthogonal tuple of rank $m$, where $m=2+\rk(T_1-1)\in \{5,6\},$
       we obtain an irreducible quadruple of rank 4 with 3 transvections, 
        using the identities
      \begin{eqnarray*} \Lambda^2(\Sp_4(\CC)) = \SO_5(\CC),&&  \Lambda^2(\SL_4(\CC)) = \SO_6(\CC).\end{eqnarray*} 
      But this contradicts the Scott formula.
\end{enumerate}

\qed\end{proof}

\subsubsection{The $P_4$ case}

\begin{thm}
 In all the $P_4$ cases
  a  symplectically rigid  tuple $\T$ can be reduced via middle convolution operations and
taking  tensor products and rational pullbacks to a rank $1$ tuple.
\end{thm}\newpage

\begin{proof}
\begin{enumerate}
 \item In the case $P_4(8,8,10,10)$ the dimension count contradicts the irreducibility.
 
\item A tuple $\T$ in the  $P_4(6,8,10,10)$ case is linearly rigid.

\item In the case $P_4(6,6,10,10)$ the irreducibility of $\T$ implies
      that $\rk(T_4+1)=1$.
      Hence $\S=\MC_{-1}(\T)$ is an orthogonal rank $2$ tuple having two involutions. The claim follows as in the proof (ii) of Thm.~\ref{P3}.

\item  A tuple $\T$ in the  $P_4(8,8,8,10)$ case is linearly rigid.

\item In the case $P_4(6,8,8,10)$
      the tuple $\S=\MC_{-1}(\T)$ is an orthogonal tuple of rank $5$.
      A suitable sequence as in Prop.~\ref{eigen2} (iii) yields
      an orthogonal  tuple of rank $2$ .
       The claim follows as in the proof (ii) of Thm.~\ref{P3}.

\item The case $P_4(6,6,8,10)$ is excluded by the Scott
      formula.

 \item In the case $P_4(8,8,8,8)$ Scott's lemma shows that
       $\Lambda^2 (\T)$ has a three dimensional orthogonal composition factor.
       By Cor~\ref{la2} we get that
        $\T$ is a tensor product of two quadruples of rank $2$  containing a trivial element.
       Hence we are in the linearly rigid case.

 \item In the case $P_4(6,8,8,8)$     
       we get that  $\S=\Lambda^2 (\T)$ is reducible.
       The Scott formula and Cor.~\ref{la2} imply that
       $(S_1,S_2,-S_3,-S_4) $ splits into a trivial $1$ dimensional
       component and a $4$ dimensional one. 
       Since the rank $4$ tuple is linearly rigid the claim follows.

\item In the case $P_4(6,6,8,8)$
      $\MC_{-1}(\T)$ 
      is an orthogonal rank $4$ tuple in $\SO_4(\CC)^4$, where
      $\J(T_3)=\J(T_4)=(\J(2),\J(2))$.
      Thus we can decompose it into a tensor product
      of two linearly rigid rank $2$ tuples.
\end{enumerate}

\qed\end{proof}

\subsubsection{The $P_5$ case}

\begin{thm}
 In all $P_5$ cases a symplectically rigid tuple $\T$ can be reduced via middle convolution operations,
taking  tensor products and rational pullbacks to a rank $1$ tuple.
\end{thm}

\begin{proof}
\begin{enumerate}
 \item In the case $P_5(10,10,10,10,10)$ the Scott formula implies
      that 
       \[\J(\T)=((\J(2),1,1),(\J(2),1,1), (\J(2),1,1),(\J(2),1,1),(-\J(2),-1,-1)).\]
 Thus the tuple is linearly rigid, a so called Jordan-Pochhammer tuple.

\item In the $P_5(8,10,10,10,10)$ case we get a contradiction to the Scott formula.

\item The   $P_5(8,8,10,10,10)$ case is  ruled out by
      the Scott formula.

\item In the case $P_5(8,8,8,10,10)$ the application of $\MC_{-1}$  yields
      an orthogonal rank $4$ tuple with Jordan forms
      \[ ((\J(2),\J(2)),\quad (\J(2),\J(2)),\quad (\J(2),\J(2)),\quad (-1,1,1,1),\quad (-1,1,1,1)).\]
      Hence a quadratic pullback can be written as a tensor product of two
       linearly rigid six tuples of  rank $2$  with non trivial Jordan forms
      $(\J(2),\J(2),-\J(2))$ each.

\item We can rule out the case $P_5(8,8,8,8,10)$.
      Otherwise $\S=\MC_{-1}(\T)$ yields  an orthogonal rank $5$ tuple with Jordan forms $\J(S_1)=\ldots =\J(S_4)=(\J(2),\J(2),1)$ and $\J(S_5)=(-1,-1,-1,-1,1)$.
      Using $\Lambda^2\Sp_4 =  \SO_5$ we get a symplectic rank $4$ tuple with
      Jordan forms
 \[ ((\J(2),1,1),\quad (\J(2),1,1),\quad (\J(2),1,1),\quad (\J(2),1,1),\quad (-1,-1,1,1)).\]
      But this contradicts the Scott formula.

\item In the case $P_5(8,8,8,8,8)$ we apply $\MC_{-1}$ and obtain
      an orthogonal tuple  $\S$ of rank $6$ with Jordan forms 
       $\J(S_1)=\ldots =\J(S_4)=-\J(S_5)=(\J(2),\J(2),1,1)$.
      Since $\Lambda^2\SL_4(\CC) = \SO_6(\CC)$ we get a tuple of  rank $4$  with Jordan forms
 \[ ((\J(2),1,1),\quad   (\J(2),1,1),\quad    
 (\J(2),1,1),\quad   (\J(2),1,1),\quad   \pm (i\J(2),i,i)).\]
     The linear rigidity yields the claim.
\end{enumerate}

\qed\end{proof}

\begin{rem}
 In the $P_5(8,8,8,8,8)$ case the monodromy group $G=\langle \T \rangle$ is a finite $2$-group of order $32$, where $Z(G)=G'$ and $G/G'\cong Z_2^4$.
\end{rem}

\section{Translation to differential operators}

Let as usual $\frac{d}{dz}$ be the derivation on $\mathbb{C}[z]$ defined by $\frac{d}{dz}(z)=1$ and $\mathbb{C}[z,\partial]:=\mathbb{C}[z][\partial]$ be the ring of differential operators with respect to $\frac{d}{dz}$. An element $P\in\mathbb{C}[z,\partial]$ with singular locus $S\subset \mathbb{C}\cup\{\infty\}$ can be regarded as a linear homogeneous differential equation on $\MP^1\setminus S$. Thus, we can investigate its induced local system $\mathbb{L}$ on $\mathbb{P}^1\setminus S$ with respect to the following conventions. 
\vspace{2ex}

\begin{conv}
We fix once and for all an orientation on $\mathbb{P}^1$ and denote the winding number of a closed path $\gamma$ around a point $p\in\mathbb{P}^1\setminus \im (\gamma)$ by $\nu_{\gamma}(p)$. Furthermore, we denote the singular locus of a differential operator $L\in\mathbb{C}[z,\partial]$ by $S$, if this leads to no confusion. Having chosen an arbitrary base point $x_0\in\mathbb{P}^1\setminus S$, we attach to each $p\in\mathbb{P}^1$ a loop $\gamma_p$ starting at $b$ with $\nu_{\gamma}(p)=1$ and $\nu_{\gamma}(s)=0$ for all $s\in S\setminus\{p\}$. Then $\{\gamma_s\}_{s\in S}$ is a set of generators of $\pi_1\left(\mathbb{P}^1\setminus S,x_0\right)$ and we equip $S$ with an ordering $S=\{s_1,\dots,s_{r+1}\}$ such that their composition $\prod_{i=1}^{r+1}\gamma_{s_i}$ is homotopic to the trivial loop. We set the \textbf{monodromy tuple} associated to $L$ to be \[\T:=\left(T_1,\dots,T_{r+1}\right):=\left(\rho_{\mathbb{L}}\left(\gamma_{s_1}\right),\dots,\rho_{\mathbb{L}}\left(\gamma_{s_{r+1}}\right)\right)\in \GL(\mathbb{L}_{x_0})^{r+1}.\] 
\end{conv}
\vspace{2ex}

We translate the constructions for monodromy tuples used before to the level of differential operators in an appropriate way.
Mainly for computational and aesthetical reasons we use the so called \textit{logarithmic derivation} $z\frac{d}{dz}$ on $\mathbb{C}[z]$ and the ring of differential operators $\mathbb{C}[z,\vartheta]:=\mathbb{C}[z][\vartheta]$ with respect to $z\frac{d}{dz}$, which can naturally be regarded as a subring of $\mathbb{C}[z,\partial]$. We call an operator $L=\sum_{i=0}^na_i\vartheta^i$ with $a_i\in\mathbb{C}[z]$ \textit{reduced}, if the greatest common divisor of all its coefficients $a_i$ is a unit. The \textit{degree} $\deg(L)$ of $L$ is the maximal $i$ for which $a_i\not=0$. Rearranging the coefficients, we also may write $L=\sum_{i=0}^mz^iP_i$, with $P_i\in\mathbb{C}[\vartheta]$. Recall, that $P_0$ is the \textit{indicial equation} of $L$ at $z=0$ and the roots of $P_0$ - considering $\vartheta$ as a formal variable - are the \textit{exponents} $E$ of $L$. For each exponent $e$, we have a formal solution $f\in z^{\mu}\mathbb{C}\llbracket z\rrbracket^{*}$ of $L$ at $z=0$, where $\mu\in \left(e+\mathbb{N}_0\right)\cap E$. We call $\mu$ the \textit{exponent} of the solution $f$. The indicial equation and the exponents of $L$ at the other points $p\in\mathbb{P}^1$ can be obtained in the same way after having performed the transformation $z\mapsto z+p$ or $z\mapsto \frac{1}{z}$. We call $L$ \textit{fuchsian}, if the degree of its indicial equation at each point $p\in\mathbb{P}^1$ equals $\deg(L)$. This agrees with the usual definition of a fuchsian operator as given in \cite[Section 6.2]{Put}. As by Deligne's investigations in \cite{Del} each operator of geometric origin has to be fuchsian, we will perform all constructions with operators of this type. 
\vspace{2ex}

All local systems in the constructions done before are built up from local systems of the form
\[\Lambda_{\alpha}=\left(1,\alpha^{-1}, \alpha\right)\] for $a\in\MQ$ and $\alpha=\exp(2\pi ia)$ with respect to the points $\{0,1,\infty\}$. Thus the basic operators we are dealing with are those of order one, which induce this monodromy tuple.

\begin{defin}
Let $a\in\MQ$. We set \[L_a:=\vartheta-z(\vartheta+a)\in\mathbb{C}[z,\vartheta].\]
\end{defin}

\begin{rem}
The solution space of $L_a$ is spanned by the formal expression  
\[f=\frac{1}{(1-z)^a},\] which is algebraic over $\mathbb{Q}(z)$. 
Thus $L_a$ is of geometric origin and its induced monodromy tuple is precisely $\Lambda_{\alpha}$. Two operators $L_a$ and $L_b$ induce the same monodromy tuple if and only if $a-b\in\MZ$. 
\end{rem}

\subsection{Tensor product}
We state the definition of the tensor product of differential operators as it is given in \cite[Chapter 2]{Put} and investigate some basic properties. Let us briefly recall that there is a universal Picard-Vessiot ring $\mathcal{F}$ of $\left(\mathbb{C}[z],z\frac{d}{dz}\right)$, i.e. for each $L\in\mathbb{C}[z,\vartheta]$ the set $\Sol_L:=\{y\in\mathcal{F}\mid L(y)=0\}$ can be regarded as a $\deg(L)-$dimensional $\mathbb{C}$ vectorspace. Therefore we call $\Sol_L$ the \textit{solution space} of $L$.

\begin{defin}
Let $L_1,L_2\in \mathbb{C}[z,\vartheta]$ be reduced. The \textbf{tensor product} $L_1\otimes L_2\in \mathbb{C}[z,\vartheta]$ of $L_1$ and $L_2$ over $\mathbb{C}[z]$ is the reduced operator of minimal degree, whose solution space contains the set $\{y_1y_2\mid L_1(y_1)=L_2(y_2)=0\}\subset\mathcal{F}$.
\end{defin}

\begin{rem}
\begin{enumerate}
 \item We always have $L_1\otimes L_2\in\mathbb{C}[z,\vartheta]$, as the vector space $V\subset \mathcal{F}$ spanned by $\{y_1y_2\mid L_1(y_1)=L_2(y_2)=0\}$ is set-wise invariant
under the natural action of the differential Galois group $G$ of $\mathcal{F}\supset\mathbb{C}[z]$. Thus by \cite[Lemma 2.17]{Put} the solution space of $L_1\otimes L_2$ is exactly $V$. 

 \item We have $\deg(L_1\otimes L_2)\leq \deg(L_1)\deg(L_2)$.

\item If $L_2$ has order one and its solution space is spanned by $g\in\mathcal{F}$, the solution space of the tensor product $L_1\otimes L_2$ is spanned by $\{gy\mid L_1(y)=0\}\subset \mathcal{F}$. Thus we write
\[L_1^{\frac{1}{g}}:=L_1\otimes L_2\in \mathbb{C}[z,\vartheta].\]

\item Symmetric and exterior powers of differential operators are defined similarly. For a reduced $L\in\mathbb{C}[z,\vartheta]$ we set $\Sym^n(L)$ to be the reduced operator of minimal degree whose solution space is spanned by the set
\[\{y_1\cdot\dots\cdot y_n\mid L(y_i)=0 \textrm{ for all } i=1,\dots,n\}\subset \mathcal{F}\] and $\Lambda^n(L)$ to be the reduced operator of minimal degree whose solution space is spanned by the set
\[\{\Wr(y_1,\dots,y_n)\mid  L(y_i)=0 \textrm{ for all } i=1,\dots,n\}\subset \mathcal{F},\] where $\Wr$ denotes the Wronskian \[\Wr(y_1,\dots,y_n):=\det\begin{pmatrix}y_1&\dots&y_n\\z\frac{d}{dz}y_1&\dots&z\frac{d}{dz}y_n\\\vdots&\vdots&\vdots\\\left(z\frac{d}{dz}\right)^{n-1}y_1&\dots&\left(z\frac{d}{dz}\right)^{n-1}y_n\end{pmatrix}\] with respect to the unique extension of $z\frac{d}{dz}$ to $\mathcal{F}$.

\end{enumerate}
\end{rem}

Since the solution space of $L_1\otimes L_2$ is locally isomorphic to a subspace of the tensor product of the solution spaces of $L_1$ and $L_2$, we have the following 

\begin{prop}
 Let $L_1,L_2\in\mathbb{C}[z,\vartheta]$ be irreducible with singular loci $S_1,S_2\in\mathbb{C}\cup\{\infty\}$ and induced monodromy tuples $\T_1$ and $\T_2$ with respect to $b\in\MP^1\setminus\{S_1\cup S_2\}$. Then the following hold. 
\begin{enumerate}
 \item The monodromy tuple induced by $L_1\otimes L_2$ is a direct summand of $\T_1\otimes \T_2$. 
\item The monodromy tuple induced by $\Sym^nL_1$ is a direct summand of $\Sym^n\T_1$. 
\item The monodromy tuple induced by $\Lambda^nL_1$ is a direct summand of $\Lambda^n\T_1$. 
\end{enumerate}
\end{prop}

We especially get

\begin{cor}
Let $L\in\mathbb{C}[z,\vartheta]$ be a monic differential operator with induced monodromy tuple $\T$, $a\in \MQ\setminus\MZ$ and $\alpha=\exp(2\pi ia)$. Then the monodromy tuple induced by $L^{(1-z)^a}=L\otimes L_a$ is precisely $\MT_{\Lambda_{\alpha}}(\T)$.  
\end{cor}

\subsection{Convolution and Hadamard product}

In this section we investigate the Hadamard product with local systems of type $\Lambda_{\alpha}$, where $\alpha\in S^1$, using relations to the convolution with certain local systems of rank one. We rather work with the Hadamard product than with the convolution on the level of differential operators.

We first define for $a\in\mathbb{Q}\setminus\mathbb{Z}$ the convolution of solutions of a fuchsian operator with $z^a$ and the Hadamard product with $(1-z)^{-a}$, which spans $\Sol_{L_a}$. 

\begin{defin}
Let $L\in\mathbb{C}[z,\vartheta]$ be fuchsian, $f$ a solution of $L$ and $a\in\MQ\setminus\MZ$.
\begin{enumerate}
 \item For two loops $\gamma_p, \gamma_q$ with $\nu_{\gamma_p}(q)=\nu_{\gamma_q}(p)=0$ we define the \textbf{Pochhammer contour} \[[\gamma_p,\gamma_q]:=\gamma_p^{-1}\gamma_q^{-1}\gamma_p\gamma_q.\]
\item  For $p\in\MP^1$, the expression \[C^{p}_a(f):=\int_{[\gamma_{p},\gamma_{z}]}f(x)(z-x)^a\frac{dx}{z-x}\] is called the \textbf{convolution} of $f$ and $z^a$ with respect to the Pochhammer contour $[\gamma_{p},\gamma_{z}]$.
 \item For $p\in\MP^1$, the expression \[H^{p}_{a}(f):=\int_{[\gamma_{p},\gamma_{z}]}f(x)\left(1-\frac{z}{x}\right)^{-a}\frac{dx}{x}\] is called the \textbf{Hadamard product} of $f$ and $(1-z)^{-a}$ with respect to the Pochhammer contour $[\gamma_{p},\gamma_{z}]$.
\end{enumerate}
 \end{defin}

\begin{rem}\label{ConvtoHada}

\begin{enumerate}
\item In the sequel, we will frequently use the following formulae for integrals involving Pochhammer contours for $z\not\in S$:
\begin{enumerate}
 \item  \[\int_{\gamma_p\gamma_q}f(x)dx=\int_{\gamma_q}f(x)dx+\int_{\gamma_p}\rho_{\mathcal{L}}(\gamma_q)\left(f\right)(x)dx.\]
\item \[\int_{[\gamma_p\gamma_q,\gamma_{z}]}f(x)(\lambda-x)^a\frac{dx}{\lambda-x}=C_a^q(f)+C_a^p\left(\rho_{\mathcal{L}}(\gamma_q)(f)\right).\]
\item \[\int_{\left[\gamma_p^{-1},\gamma_z\right]}f(x)(z-x)^a\frac{dx}{z-x}=-C_a^p\left(\rho_{\mathcal{L}}(\gamma_p)^{-1}(f)\right).\]
\end{enumerate}

\item If $f\in(z-p)^{\mu}\mathbb{C}\llbracket z-p\rrbracket$ near $z=p$, we get
\[C^{p}_a(f)=(1-\exp(2\pi i \mu))\int_{\gamma_z}f(x)(z-x)^a\frac{dx}{(z-x)}+(\exp(2\pi ia)-1)\int_{\gamma_p}f(x)(z-x)^a\frac{dx}{(z-x)}.\] 
In particular, we have
\[\int_{\gamma_z}f(x)(z-x)^a\frac{dx}{(z-x)}=(1-\exp(2\pi i a))\int_{x_0}^{z}f(x)(z-x)^a\frac{dx}{(z-x)}\] and 
\[\int_{\gamma_p}f(x)(z-x)^a\frac{dx}{(z-x)}=(1-\exp(2\pi i \mu))\int_{x_0}^pf(x)(z-x)^a\frac{dx}{(z-x)},\] if $\mu$ is not a negative integer. Thus we get 
\[C^{p}_a(f)=(1-\exp(2\pi i \mu))(1-\exp(2\pi i a))\int_{p}^zf(x)(z-x)^a\frac{dx}{(z-x)}.\] Note that the right hand side does not depend on the choice of the base point $x_0\in\mathbb{P}^1\setminus S$ and may be interpreted as a meromorphic function near $z=p$.

\item One checks that the convolution and the Hadamard product for a fixed Pochhammer contour $[\gamma_p,\gamma_z]$ are related by the following formulae
\begin{enumerate}
\item $C^{p}_a(f)=(-1)^{a-1}H^{p}_{1-a}(z^af)$.
\item $H^{p}_a(f)=(-1)^{-a}C^{p}_{1-a}\left(z^{a-1}f\right)$.
\end{enumerate} 
\end{enumerate}
\end{rem}

In order to find differential equations having solutions $C^p_a(f)$, we investigate some properties of the convolution. 

\begin{lemma}\label{rules}
Let $L\in\mathbb{C}[z,\vartheta]$ be fuchsian, $f$ a solution of $L$, $a\in\MQ\setminus\MZ$, $p\in\mathbb{P}^1$ and $[\gamma_p,\gamma_z]$ a fixed Pochhammer contour. We have the following relations
\begin{enumerate}
 \item $\frac{d}{dz}C^p_a(f)=C^p_{a}\left(\frac{d}{dz}f\right)=(a-1)C^p_{a-1}(f)$.
\item $C^p_a(zf)=zC^p_a(f)-C^p_{a+1}(f)$.
\item $C_a^p\left(z\frac{d}{dz}f\right)=(z\frac{d}{dz}-a)C_a^p(f)$.
\item $C_a^p\left(z^if\right)=\prod_{j=0}^{i-1}\left(\frac{z\frac{d}{dz}}{a+j}-1\right)C_{a+i}^p(f)$.
\end{enumerate}
\end{lemma}

\begin{proof}
Using Leibniz's rule for differentiating under the integral sign we get 
\begin{align*}
 \frac{d}{dz}\int_{[\gamma_p,\gamma_z]}f(x)(z-x)^{a-1}dx=\int_{[\gamma_p,\gamma_z]}f(x)\frac{d}{dz}(z-x)^{a-1}dx=-\int_{[\gamma_p,\gamma_z]}f(x)\frac{d}{dx}(z-x)^{a-1}dx.
\end{align*}
As the monodromy of $f(x)(z-x)^{a-1}$ along $[\gamma_p,\gamma_z]$ is trivial, integration by parts yields
\begin{align*}
-\int_{[\gamma_p,\gamma_z]}f(x)\frac{d}{dx}(z-x)^{a-1}dx=\int_{[\gamma_p,\gamma_z]}\left(\frac{d}{dx}f(x)\right)(z-x)^{a-1}dx
\end{align*}
and hence the first result. The other results are obtained by direct computation and the results established before.\qed
\end{proof}

Using those properties we get the following

\begin{prop}\label{Falt}
Let $L=\sum_{i=0}^mz^iP_i(\vartheta)\in\mathbb{C}[z,\vartheta]$ be fuchsian, $f$ a solution of $L$ and $a\in\MQ\setminus\MZ$. Then $C^p_a(f)$ is a solution of
\[\mathcal{C}_a(L):=\sum_{i=0}^mz^i\prod_{j=0}^{i-1}(\vartheta+i-a-j)\prod_{k=0}^{m-i-1}(\vartheta-k)P_i(\vartheta-a)\] for each $p\in\mathbb{P}^1$. 
\end{prop}

\begin{proof}
For $0\leq i\leq m$ and $b\in\MQ\setminus\MZ$ we have 
\begin{align*}
 C^p_{b+i}(g)&=\frac{1}{\prod_{l=1}^{m-i}(b+m-l)}\left(C^p_{b+m}(g)\right)^{(m-i)}=z^{i-m}\frac{z^{m-i}}{\prod_{l=1}^{m-i}(b+m-l)}\left(C^p_{b+m}(g)\right)^{(m-i)}\\&=z^{i-m}\frac{\prod_{k=0}^{m-i-1}(\vartheta-k)}{\prod_{l=1}^{m-i}(b+m-l)}C^p_{b+m}(g). 
\end{align*}
for each $g$ which is a solution of some $R\in\mathbb{C}[z,\vartheta]$ by Lemma \ref{rules}. Thus 
\begin{align*}
0&=C^p_{b}\left(Lf\right)=\sum_{i=0}^m C^p_b\left(z^iP_i(\vartheta)f\right)=\sum_{i=0}^{m}\prod_{j=0}^{i-1}\left(\frac{\vartheta}{b+j}-1\right)C^p_{b+i}\left(P_i(\vartheta)f\right)\\&=
\sum_{i=0}^{m}\prod_{j=0}^{i-1}\left(\frac{\vartheta}{b+j}-1\right)z^{i-m}\frac{\prod_{k=0}^{m-i-1}(\vartheta-k)}{\prod_{l=1}^{m-i}(b+m-l)}C^p_{b+m}\left(P_i(\vartheta)f\right)\\&=\sum_{i=0}^{m}z^{i-m}\prod_{j=0}^{i-1}\left(\frac{\vartheta+i-m}{b+j}-1\right)\frac{\prod_{k=0}^{m-i-1}(\vartheta-k)}{\prod_{l=1}^{m-i}(b+m-l)}P_i(\vartheta-(b+m))C^p_{b+m}\left(f\right)\\&=\frac{1}{z^m\prod_{i=0}^{m-1}(b+i)}\sum_{i=0}^mz^i\prod_{j=0}^{i-1}(\vartheta+i-m-b-j)\prod_{k=0}^{m-i-1}(\vartheta-k)P_i(\vartheta-(b+m))C^p_{b+m}\left(f\right).
\end{align*}
Setting $b=a-m$, we get the desired result.\qed
\end{proof}

An approach via so called Euler-integrals can be found in \cite[Chapter II.3]{IKShY} and yields a similar operator in $\mathbb{C}[z,\partial]$. We use the relations between the convolution and the Hadamard product to obtain an operator having solutions of the form $H^p_a(f)$.

\begin{cor}\label{FHada}
Let $L=\sum_{i=0}^m z^iP_i\in\mathbb{C}[z,\vartheta]$ be fuchsian, $f$ a solution of $L$ and $a\in\mathbb{Q}\setminus \mathbb{Z}$. Then $H^p_a(f)$ is a solution of \[\mathcal{H}_a(L):=\sum_{i=0}^mz^i\prod_{j=0}^{i-1}(\vartheta+a+j)\prod_{k=0}^{m-i-1}(\vartheta-k)P_i\] for each $p\in\mathbb{P}^1$. 
\end{cor}

Note that for an arbitrary fuchsian operator $L$ the monodromy tuple induced by $\mathcal{C}_a(L)$, resp. $\mathcal{H}_a(L)$, is a subfactor of $\MC_{\alpha}(\T)$, resp. $\MH_{\alpha^{-1}}(\T)$. To induce the tuple $\MC_{\alpha}(\T)$ we will restrict ourselves to operators, for which the expression $f(z)(y-z)^{a-1}$ is free of residues with respect to every $y\in\mathbb{P}^1$. This is guaranteed, if the operator $L$ is positive in the following sense.

\begin{defin}
Let $a\in\mathbb{Q}\setminus\mathbb{Z}$. A differential operator $L\in\mathbb{C}[z,\vartheta]$ is called \textbf{a-positive}, if $L$ is fuchsian, has no exponents in $\mathbb{Z}_{<0}$ at each point $p\in\mathbb{C}$ and no exponents in $-a+\mathbb{Z}_{\leq 0}$ at $p=\infty$. 
\end{defin}

The next proposition justifies, that there is an operator in $\mathbb{C}[z,\vartheta]$, whose solution space is spanned by all $C^p_a(f)$, where $f$ is a solution of an $a-$positive operator $L$ and that this operator induces the desired monodromy tuple. As we have $C^p_a(f)=0$ if $f$ is holomorphic at $p$ by Remark \ref{ConvtoHada}, we can concentrate ourselves on the expressions $C^s_a(f)$ for $s\in S$.

\begin{prop}\label{OpundMon}
Let $a\in\MQ\setminus\MZ$, $L\in\mathbb{C}[z,\vartheta]$ be irreducible, $a$-positive with $\deg(L)=n$, $S=\{s_1,\dots,s_r,\infty\}$ and $\alpha=\exp(2\pi ia)$. Let furthermore $\{f_1,\ldots,f_n\}$ be a basis of $\Sol_L$, \[R=\left(C^{s_1}_a(f_1),\dots,C^{s_1}_a(f_n),\dots,C^{s_r}_a(f_1),\dots,C_a^{s_r}(f_n)\right)\] and
\[V:=\left\{R \cdot v \mid v \in \CC^{nr}\right\}. \] Then the action of $C_{\alpha}(\T)$ on $V$ as described in Section 2.2 is given by $\MC_{\alpha}(\T)$.
\end{prop}

\begin{proof}
 Due to \cite[Section 4]{DR07} the vector space $V$
 is invariant under the action of the monodromy $C_{\alpha}(\T)$. 
Let $F=(f_1,\dots,f_n)$, $\mathcal{K}_k$ and $\mathcal{L}$ as in Section \ref{katz} and $v=(v_1,\dots,v_r)^{tr}$, where $v_i\in\mathbb{C}^n$.
 Since $L$ is $a$-positive,
 $F\cdot v_k$ is holomorphic at $s_k$ for $v_k \in \ker(T_{s_k}-\id)$ and
 we get $R \cdot v =0$
 for $v \in \K_k$. Thus we have \[\dim_{\mathbb{C}}(V)\leq\sum_{s\in S\setminus\{\infty\}} \rank(T_s-\id).\]
 
We can choose for each $z\in\mathbb{P}^1\setminus S$ a path $\gamma_{z}$ fulfilling our conventions such that \[\gamma_{s_1}\cdots\gamma_{s_r}\gamma_{z}\gamma_{\infty}=1.\] With respect to the basis $F$ of $\Sol_L$ and letting $C_a^p$ operate on each component of $F$, the elements of the monodromy group of $L$ operate via
\[C_a^p\left(\rho_{\mathbb{L}}\left(\gamma_{s_i}\right)\left(F\cdot v\right)\right)=C_a^p(F)\cdot T_{s_i}v=C_a^p\left(F\cdot T_{s_i}v\right),\] for each $v\in\mathbb{C}^n$ and each $0\leq i\leq r$. Furthermore, by definition the induced monodromy action of the path $\gamma_{z}$ on the integrand of $C_a^p(f)$ is just given by multiplication with $\alpha$.
Using the rules established before, we have
\begin{align*}
\int_{[\gamma_{\infty}^{-1},\gamma_z]}F\cdot v(z-x)^a\frac{dx}{z-x}=\int_{[\gamma_{s_1}\cdots\gamma_{s_r}\gamma_z,\gamma_z]}F\cdot v(z-x)^a\frac{dx}{z-x}=\sum_{i=1}^rC_a^{s_i}\left(F\right)\cdot T_{s_{i+1}}\cdots T_{s_r}\alpha v 
\end{align*}
on the one hand and 
\begin{align*}
\int_{[\gamma_{\infty}^{-1},\gamma_z]}F\cdot v(z-x)^a\frac{dx}{z-x}=-C_a^{\infty}(F)\cdot \alpha T_{\infty}^{-1}v=-C_a^{\infty}(F)\cdot \alpha T_{s_1}\cdots T_{s_r} v
\end{align*}
on the other and thus the relation
\begin{align*}
C_a^{\infty}(F)\cdot \alpha T_{s_1}\cdots T_{s_r}v=-\sum_{i=1}^rC_a^{s_i}\left(F\right)\cdot\alpha T_{s_{i+1}}\cdots T_{s_r} v.  
\end{align*}
for each $v\in\mathbb{C}^n$. 
As the left hand side is zero for $v_{\infty}\in\ker\left(\alpha T_{s_1}\cdots T_{s_r}-\id\right)$, rewriting the right hand side yields
$R\cdot v=0$ for each $v\in \mathcal{L}$.

 Hence we get
 \[\dim_{\mathbb{C}}V\leq \sum_{s\in S\setminus\{\infty\}} \rank(T_s-\id)-(n-\rank(\alpha T_{\infty}^{-1}-\id)).\]
 By the definition of $\MC_{\alpha}(\T)$ and comparing the dimensions we get the result.\qed

\end{proof}

\begin{rem}
With the notations used in the proposition above and by the relations between the convolution and the Hadamard product, assuming that $L^{z^{1-a}}$ is $(1-a)$-positive and setting \[\tilde{R}=\left(H^{s_1}_a(f_1),\dots,H^{s_1}_a(f_n),\dots, H^{s_r}_a(f_1),\dots,H_a^{s_r}(f_n)\right)\] and
\[\tilde{V}=\left\{R \cdot v \mid v \in \CC^{nr}\right\}, \] the action of $H_{\alpha^{-1}}(\T)$ on $\tilde{V}$ is given by $\MH_{\alpha^{-1}}(\T)$. 
 \end{rem}

As $C_{\alpha}(\T)$ and $H_{\alpha^{-1}}(\T)$ are induced by a fuchsian systems, their Zariski closures over $\mathbb{C}$ are isomorphic to the differential Galois groups of the corresponding systems, see e.g. \cite[Corollary 5.2]{Put}. By the preceding proposition and \cite[Lemma 2.17]{Put}, there are non trivial differential operators in $\mathbb{C}[z,\vartheta]$ whose solution spaces are exactly $V$, resp. $\tilde{V}$. This justifies the following definition.

\begin{defin}
Let $a\in\mathbb{Q}\setminus\mathbb{Z}$ and $L\in\mathbb{C}[z,\vartheta]$ be irreducible. 
\begin{enumerate}
 \item If $L$ is $a$-positive, the \textbf{convolution} $L\star_C \left(\vartheta -a\right)$ of $L$ and $\vartheta-a$ is the non trivial reduced operator of minimal degree in $\mathbb{C}[z,\vartheta]$ whose solution space contains the set \[\bigcup_{p\in\mathbb{P}^1}\{C^{p}_a(f)\mid f \textrm{ is a solution of } L\}.\]
\item If $L^{z^{1-a}}$ is $(1-a)$-positive, the \textbf{Hadamard product} $L\star_H L_a$ of $L$ and $L_a$ is the non trivial reduced operator of minimal degree in $\mathbb{C}[z,\vartheta]$ whose solution space contains the set \[\bigcup_{p\in\mathbb{P}^1}\{H^{p}_{a}(f)\mid f \textrm{ is a solution of } L\}.\]
\end{enumerate}
\end{defin}

As a consequence of Proposition \ref{OpundMon}, we get

\begin{cor}\label{rankofconv}
Let $L\in\mathbb{C}[z,\vartheta]$ be irreducible with $\deg(L)=n$ and singular locus $S$. Let furthermore $S=\{0,s_2,\dots,s_r,\infty\}$, $a\in\MQ\setminus\MZ$ and $\alpha=\exp(2\pi ia)$.  
\begin{enumerate}
 \item If $L$ is $a$-positive, $L\star_C(\vartheta-a)\in\mathbb{C}[z,\vartheta]$ is an irreducible fuchsian right factor of $\mathcal{C}_a(L)$ of degree \[\deg\left(L\star_C(\vartheta-a)\right)=\sum_{s\in S\setminus\{\infty\}}\rank(T_s-\id)-\left(n-\rank\left(\alpha^{-1} T_{\infty}-\id\right)\right).\] Furthermore, its induced monodromy tuple is $\MC_{\alpha}(\T)$.
\item If $L^{z^{-a}}$ is $(1-a)$-positive, $L\star_H L_a\in\mathbb{C}[z,\vartheta]$ is an irreducible fuchsian right factor of $\mathcal{H}_a(L)$ of degree
\[\deg(L\star_HL_a)=\sum_{s\in S\setminus\{0\}}\rank(T_s-\id)-(n-\rank(\alpha T_0-\id)).\] Furthermore, its induced monodromy tuple is $\MH_{\alpha^{-1}}(\T)$.
\end{enumerate}

\end{cor}

The degree of the operator $\mathcal{C}_a(L)$, resp. $\mathcal{H}_a(L)$, is possibly much higher than the degree of $L\star_C(\vartheta-a)$, resp. $L\star_HL_a$. As we know the degrees of $L\star_C(\vartheta-a)$, resp. $L\star_HL_a$, we can try to find those operators by a factorization of $\mathcal{C}_a(L)$, resp. $\mathcal{H}_a(L)$, into irreducible operators. Such a factorization is in general not unique, but yields a composition series of the solution space $W$ of the operator with respect to the action of its differential Galois group $G$, see e.g \cite[Proposition 2.11]{Singer}. It will turn out that in our cases we always have a factorization \[\mathcal{H}_a(L)=\prod_{i=0}^l(\vartheta+c_i)R,\] with $c_1,\dots,c_l\in\mathbb{C}$ and $\deg(R)=\deg(L\star_HL_a)>1$. As then the only $\deg(L\star_H L_a)$-dimensional $G$-invariant subspace of $W$ on which $G$ acts irreducibly is exactly the solution space of $R$, we have $R=L\star_HL_a$. In particular, we have the following quite technical

\begin{prop}
Let $a\in\mathbb{Q}\setminus\mathbb{Z}$, $L=\sum_{i=0}^mz^iP_i\in\mathbb{C}[z,\vartheta]$ be irreducible and $\{0,\infty\}\subset S$. Let furthermore $k_0\in\mathbb{N}$ maximal such that $\prod_{i=0}^{k_0-j-1}(\vartheta+a-1-i)$ divides $P_j$ for all $0\leq j\leq k_0-1$ and $k_{\infty}\in\mathbb{N}$ maximal such that $\prod_{i=0}^{k_{\infty}-j-1}(\vartheta+1+i)$ divides $P_{m-j}$ for all $0\leq j\leq k_{\infty}-1$. Then 
\begin{enumerate}
 \item $\mathcal{H}_a(L)=\prod_{i=0}^{k_0-1}(\vartheta+a-1-i)\prod_{j=0}^{k_{\infty}-1}(\vartheta-m+1+j)R,$ with $R\in\mathbb{C}[z,\vartheta]$.
\item If $L^{z^{-a}}$ is $(1-a)-$positive, the operator $L\star_H L_a$ is an irreducible right factor of $R$.
\item If $L^{z^{-a}}$ is $(1-a)-$positive, $m=\sum_{s\in S\setminus\{0,\infty\}}\rank(T_s-\id)$, $\rank(\exp(2\pi i a)T_0-\id)=n-k_0$ and $\rank(T_{\infty}-\id)=n-k_{\infty}$, we have $R=L\star_H L_a$.
\end{enumerate}
\end{prop}

\begin{proof}
By Corollary \ref{FHada} we have
\[\mathcal{H}_a(L)=\sum_{i=0}^mz^i\prod_{j=0}^{i-1}(\vartheta+a+j)\prod_{k=0}^{m-i-1}(\vartheta-k)P_i.\] Since $\mathcal{H}_a(L)$ has a left factor of the form $\vartheta+c$ with $c\in\mathbb{C}$ if and only if $\vartheta+c+i$ divides $\prod_{j=0}^{i-1}(\vartheta+a+j)\prod_{k=0}^{m-i-1}(\vartheta-k)P_i$ for each $0\leq i\leq m$, we obtain the first part of the statement by a direct computation. The second part is a direct consequence of Corollary \ref{rankofconv}. To prove the third part, note that we have 
\begin{align*}
 \deg(R)&=n+m-k_0-k_{\infty}=\sum_{s\in S\setminus\{0,\infty\}}\rank(T_s-\id)+\rank(T_{\infty}-\id)-(n-\rank(\alpha T_0-\id))\\&=\deg(L\star_H L_a)
\end{align*} by Corollary \ref{rankofconv}. Now the action of the Galois group on the solution space as discussed above yields the result.\qed
\end{proof}

A more general treatment of the factorization of $\mathcal{H}_a(L)$ will be discussed in a subsequent article.

\begin{ex}\label{Hadaab}
Let $a,b\in\MQ\setminus\MZ$. Recall that the monodromy tuple induced by $L_a$, where the singular locus of $L_b$ is extended by the apparent singularity $z=0$, is given by $\T=(T_0,T_1,T_{\infty})=(1,\beta^{-1},\beta)$, where $\beta=\exp(2\pi ib)$. Thus we have $\deg(L_b\star_H L_a)=2$ and
\[\mathcal{H}_b(L_a)=\vartheta^2-z(\vartheta+b)(\vartheta+a)=L_b\star_H L_a.\] 
Inductively, one shows that \[L_{a_1}\star_HL_{a_2}\star_H\dots\star_H L_{a_n}=\vartheta^n-z\prod_{i=1}^n(\vartheta+a_i).\] In particular, each of those operators is of hypergeometric type.
\end{ex}

The situation on local systems suggests, that the operation $\mathcal{H}_a$ is invertible. As we will see in the next lemma, this is not exactly the case.
\begin{lemma}\label{HadInv}
Let $L=\sum_{i=0}^m z^i P_i\in\mathbb{C}[z,\vartheta]$ and $a\in\mathbb{Q}\setminus\mathbb{Z}$. Then
\[\mathcal{H}_{1-a}\left(\mathcal{H}_a(L)^{z^{-a}}\right)=\prod_{k=1}^{m-1}(\vartheta-k)\prod_{j=0}^{m-1}(\vartheta-a-j)L^{z^{-a}}\vartheta.\]
\end{lemma}

\begin{proof}
As we have
\begin{align*}
L^{z^{-a}}=\sum_{i=0}^m z^i P_i(\vartheta-a)
\end{align*}
and 
\begin{align*}
\mathcal{H}_a(L)^{z^{-a}}=\sum_{i=0}^mz^i\prod_{j=0}^{i-1}(\vartheta+j)\prod_{k=0}^{m-i-1}(\vartheta-a-k)P_i(\vartheta-a),
\end{align*}
we obtain
\begin{align*}
\mathcal{H}_{1-a}\left(\mathcal{H}_a(L)^{z^{-a}}\right)&=\sum_{i=0}^mz^i\prod_{k=0}^{m-i-1}(\vartheta-k)\prod_{j=0}^{i-1}(\vartheta+j)\prod_{j=0}^{i-1}(\vartheta+1-a+j)\prod_{k=0}^{m-i-1}(\vartheta-a-k)P_i(\vartheta-a)\\&=\sum_{i=0}^mz^i\vartheta\prod_{k=1}^{m-1}(\vartheta-m+k+i)\prod_{j=0}^{m-1}(\vartheta-a-j+i)P_i(\vartheta-a)
\end{align*}
and hence the result.\qed
\end{proof}

Nevertheless, this lemma turns to be quite useful to determine solutions of $\mathcal{H}_a(L)$ involving logarithms as we will see in the next section.

\section{Special solutions}\label{Spsols}

The translation of the constructions appearing in Katz' algorithm to the level of differential operators enables us to compute certain local solutions of a differential operator produced by those constructions in an explicit way. To be more precise, given a fuchsian operator $L$ which is constructed by tensor and Hadamard products of differential operators of lower order, we are sometimes able to state closed formulae for the coefficients of a local solution of the form $f=(z-p)^{\mu}\sum_{m=0}^{\infty}A_m(z-p)^m\in (z-p)^{\mu}\mathbb{C}\llbracket z-p\rrbracket$ at $z=p\in\mathbb{C}$, resp. $f=t^{\mu}\sum_{m=0}^{\infty}A_mt^m\in t^{\mu}\mathbb{C}\llbracket t\rrbracket$ for $t=\frac{1}{z}$. Those solutions will be called \textit{special}. As stated in the preceding section, if $f=(z-p)^{\mu}\sum_{m=0}^{\infty}A_m(z-p)^m$ is a solution of the differential operator $L$ at $z=p$ and $g=(z-p)^{\nu}\sum_{m=0}^{\infty}B_m(z-p)^m$ is a solution of the differential operator $\tilde{L}$ at $z=p$, their Cauchy product
\[fg=(z-p)^{\mu+\nu}\sum_{m=0}^{\infty}\sum_{k=0}^mA_kB_{m-k}(z-p)^m\] is a solution of  $L\otimes\tilde{L}$ at $z=p$. Analogously, the self Cauchy product $f^2$ is a solution of $\Sym^2L$ at $z=p$ and setting $L=\tilde{L}$, the Wronskian \[\Wr(f,g)=z(z-p)^{\nu+\mu-1}\sum_{m=0}^{\infty}\sum_{k=0}^m(2k+\mu-\nu-m)A_{k}B_{m-k}(z-p)^m\] is a solution of $\Lambda^2 L$ at $z=p$. 
The situation turns out to be slightly more complicated for the middle Hadamard product $L\star_H L_a$
Classically one defines the Hadamard product $\star_H$ of two formal power series $\sum_{m=0}^{\infty}A_mz^m\in\mathbb{C}\llbracket z\rrbracket$ and $\sum_{m=0}^{\infty}B_mz^m\in\mathbb{C}\llbracket z\rrbracket$ by term-wise multiplication, i.e.
\[\sum_{m=0}^{\infty}A_mz^m\star_H\sum_{m=0}^{\infty}B_mz^m:=\sum_{m=0}^{\infty}A_mB_mz^m.\]
As the terminology suggests, given a holomorphic solution $f=\sum_{m=0}^{\infty}A_mz^m$ of $L$ near $z=0$, the expression
\[\sum_{m=0}^{\infty}(-1)^m\binom{-a}{m}A_mz^m\] should be a solution of $L\star_HL_a$ near $z=0$, as we have \[(1-z)^{-a}=\sum_{m=0}^{\infty}(-1)^m\binom{-a}{m}z^m.\] The following more general discussion will recover those solutions. 
\vspace{2ex}

At $z=p$ the eigenfunctions of the local monodromy of a fuchsian operator $L$ are elements of $(z-p)^{\mu}\mathbb{C}\llbracket z-p\rrbracket^{*}$, where $\exp(2\pi i \mu)$ is the corresponding eigenvalue. For notational convenience, we use the following
\begin{conv}
Given $E\subset \mathbb{C}$ and two functions $f,g\colon E\to \mathbb{C}$ we write 
\begin{enumerate}
\item $f \hg g$ if there is a $c\in\mathbb{C}^{*}$ such that $f(z)=cg(z)$ 
for all $z\in E$.
\item $\displaystyle{\int_{p}^zf(x) dx}$ for the integral of $f$ along the straight line $[0,1]\to E,\ t\mapsto (1-t)p+tz$, if it exists.
\end{enumerate}
\end{conv}

The relation of $C^p_a(f)$ to the line integral given in Remark \ref{ConvtoHada} yields the following

\begin{lemma}\label{ConvLinie}
 Let $f$ be an eigenfunction of the local monodromy of $L$ at $z=p\in\mathbb{C}\cup\{\infty\}$ and $\mu$ the exponent of $z^{a-1}f$ at $p$. Then we have
\[H^p_a(f)\hg\begin{cases}\displaystyle{\int_{p}^zx^{a-1}f(x)(z-x)^{-a}dx},\ \mu\not\in\MZ\\ 0,\ \mu\in\MN_0\end{cases}.\] 
\end{lemma}

\begin{proof}
The statement follows directly from Remark \ref{ConvtoHada}.\qed
\end{proof}

Recalling the well-known \textit{Beta function}
\[\mathscr{B}(p,q):=\int_{0}^1x^{p-1}(1-x)^{q-1}dx=\frac{\ga(p)\ga(q)}{\ga(p+q)},\] which is assumed to be the analytic continuation of the expression on the very right on $\mathbb{C}\setminus\{p+q\in\mathbb{Z}_{<0}\},$ a direct computation shows

\begin{lemma}\label{ConvbeiNull}
\begin{enumerate}
\item Let $f=z^{\mu}\sum_{m=0}^{\infty} A_mz^m$ be an eigenfunction of the local monodromy of $L$ at $z=0$ with exponent $\mu\not\in\MZ_{-}$. Then
\[C^0_a(f)\hg z^{\mu+a}\sum_{m=0}^{\infty}\mathscr{B}(\mu+1+m,a)A_mz^m.\] 
\item  Let $t=\frac{1}{z}$ and $f=t^{\mu}\sum_{m=0}^{\infty} A_mt^m$ be an eigenfunction of the local monodromy of $L$ at $z=\infty$ with exponent $\mu\not\in a+\MZ_{-}$. Then
\[C^{\infty}_a(f)\hg t^{\mu-a}\sum_{m=0}^{\infty}\mathscr{B}(\mu-a+m,a)A_mt^m.\] 
\end{enumerate}
\end{lemma}

\begin{proof}
\begin{enumerate}
\item By Remark \ref{ConvtoHada}, we have 
\begin{align*}
C^0_a(f)&\hg\int_{0}^z(z-x)^{a-1}f(x)dx=\sum_{m=0}^{\infty}A_m\int_{0}^z(z-x)^{a-1}x^{\mu+m}dx\\&=\sum_{m=0}^{\infty}A_mz^{\mu+m+a-1}\int_{0}^z\left(1-\frac{x}{z}\right)^{a-1}\left(\frac{x}{z}\right)^{\mu+m}dx\\&=\sum_{m=0}^{\infty}A_mz^{\mu+m+a}\int_{0}^1\left(1-s\right)^{a-1}s^{\mu+m}ds
\end{align*}
 and thus the result.
\item We obtain the result similarly to the first part starting with
\[C^{\infty}_a(f)\hg \int_{0}^tx^{-1-a}(1-xz)^{a-1}f\left(\frac{1}{x}\right)dx.\] \qed
\end{enumerate}
\end{proof}

Combining those statements yields

\begin{prop}
\begin{enumerate}\label{Loslos}
\item Let $f$ be an eigenfunction of the local monodromy of $L$ at $z=p\in\mathbb{C}$. 
Let furthermore $z^{a-1}f=(z-p)^{\mu}\sum_{m=0}^{\infty}A_m(z-p)^m$. Then
\begin{align*}H^p_a(f)\hg\begin{cases}(z-p)^{\mu+1-a}\displaystyle{\sum_{m=0}^{\infty}\mathscr{B}(\mu+1+m,1-a)A_m(z-p)^m},\ \mu\not\in\mathbb{Z}\\ 0,\ \mu\in\mathbb{N}_0\end{cases}.\end{align*} 
In particular, if $L^{z^{1-a}}$ is $(1-a)-$positive each $\mathbb{C}$-multiple of the right hand side is a solution of $L\star_H L_a$ near $z=p$.

\item Let $t=\frac{1}{z}$ and $f$ be an eigenfunction of the local monodromy of $L$ at $z=\infty$.
Let furthermore $t^{1-a}f(t)=t^{\mu}\sum_{m=0}^{\infty}A_mt^m$. Then
\begin{align*}H^{\infty}_a(f)\hg\begin{cases}t^{\mu+a-1}\displaystyle{\sum_{m=0}^{\infty}\mathscr{B}(\mu-1+a+m,1-a)A_mt^m}, \mu\notin\mathbb{Z}\\0, \mu\in\mathbb{N}_0\end{cases}.\end{align*} 
In particular, if $L^{z^{1-a}}$ is $(1-a)-$positive each $\mathbb{C}$-multiple of the right hand side is a solution of $L\star_H L_a$ near $z=\infty$.
\end{enumerate}
\end{prop}

\begin{proof}
As seen before, we have
\begin{align*}
 H^p_a(f)&\hg \int_{p}^zx^{a-1}f(x)(z-x)^{-a}dx\hg \int_{0}^{z-p}(x+p)^{a-1}f(x+p)(z-p-x)^{-a}dx
\end{align*} 
for $p\in\mathbb{C}$.
Thus the result follows from Lemma \ref{ConvbeiNull} and Lemma \ref{ConvLinie}. The case $p=\infty$ can be treated similarly.\qed
\end{proof}

\begin{rem}
If $f$ is holomorphic at $z=0$, one recovers the Hadamard product of formal power series mentioned in the introduction of the section. More generally, if $L$ has at $z=0$ a solution of the form $f=z^{\nu}\sum_{m=0}^{\infty}A_mz^m$ we get the solution 
\[g=z^{\nu}\sum_{m=0}^{\infty}\mathscr{B}(\nu+a+m,1-a)A_mz^m=z^{\nu}\sum_{m=0}^{\infty}R(m)\mathscr{B}(a+m,1-a)A_mz^m\] of $L\star_H L_a$ at $z=0$. Using Stirling's formula, one can show that $R(m)$ behaves asymptotically like $\left(\frac{\nu+a+m}{a+m}\right)^{a-1}$. 
\end{rem}

Proposition \ref{Loslos} implies that each special solution $f$ of $L$ for which $z^{a-1}f$ is not a meromorphic eigenfunction near $z=p$ induces a special solution of $L\star_H L_a$, while the solutions $g$ for which  $z^{a-1}g$ is holomorphic at $z=p$ do not contribute to the solution space of $L\star_H L_a$.
Nevertheless, the following proposition asserts that solutions of the form $\ln g+r$ with $r\in\mathbb{C}\llbracket z\rrbracket$ induce certain holomorphic solutions of $L\star_H L_a$. 

\begin{prop}\label{Convlog}
\begin{enumerate}
\item Let $L$ be irreducible and both functions \[z^{a-1}f=(z-p)^{\mu}\sum_{m=0}^{\infty}A_m(z-p)^m\] and $z^{a-1}g$ holomorphic at $z=p$. Let furthermore $\ln$ be a branch of the logarithm at $z=0$, $\ln(z-p)f+g$ a solution of $L$ at $z=p$ and $a\in\mathbb{Q}\setminus\mathbb{Z}$. Then 
\[H_a^p(f)\hg(z-p)^{\mu+1-a}\displaystyle{\sum_{m=0}^{\infty}\mathscr{B}(\mu+1+m,1-a)A_m(z-p)^m}.\] 

\item Let $L$ be irreducible $t=\frac{1}{z}$ and both functions \[t^{1-a}f=t^{\mu}\sum_{m=0}^{\infty}A_mt^m\] and $t^{1-a}g$ holomorphic at $t=0$.
 Let furthermore $\ln$ be a branch of the logarithm at $t=0$, $\ln f+g$ a solution of $L$ at $t=0$ and $a\in\mathbb{Q}\setminus\mathbb{Z}$. Then
\[H_a^{\infty}(f)=t^{\mu+a-1}\displaystyle{\sum_{m=0}^{\infty}\mathscr{B}(\mu-1+a+m,1-a)A_mt^m}.\] 
\end{enumerate}
\end{prop}

\begin{proof}
Let $\tilde{f}=z^{a-1}f$ and $\tilde{g}=z^{a-1}g$. As the formal monodromy of $\ln(z-p)$ around $\gamma_p$ is given by $\ln(z-p)+2\pi i$, evaluating $H^p_a\left(\ln(z-p) f+g\right)$ yields
\begin{align*}
 &H^p_a\left(\ln(z-p) f+g\right)\hg C^p_{1-a}\left(\ln(z-p)\tilde{f}+\tilde{g}\right)\\&=\int_{[\gamma_p,\gamma_z]}\ln(x-p)\tilde{f}(x)(z-x)^{-a}dx\\&
=-2\pi i\int_{\gamma_z}\tilde{f}(x)(z-x)^{-a}dx+(\exp(-2\pi ia)-1)\int_{\gamma_p}\ln(x-p)\tilde{f}(x)(z-x)^{-a}dx\\&=-2\pi i (1-\exp(-2\pi ia))\int_{b}^z\tilde{f}(x)(z-x)^{-a}dx-2\pi i (\exp(-2\pi i a)-1)\int_{b}^{p}
\tilde{f}(x)(z-x)^{-a}dx\\&=-2\pi i (1-\exp(-2\pi ia))\int_{p}^z\tilde{f}(x)(z-x)^{-a}dx,
\end{align*}
hence the result by Lemma \ref{ConvLinie} and Lemma \ref{ConvbeiNull}.
The second case can be treated analogously.\qed
\end{proof}

Combining this result with Lemma \ref{HadInv}, we get

\begin{lemma}\label{coni}
Let $L\in\mathbb{C}[\vartheta,z]$, $a\in\mathbb{Q}\setminus\mathbb{Z}$, $p\not\in\{0,\infty\}$, $f=(z-p)^{\mu}\sum_{m=0}^{\infty}A_m(z-p)^m\in\mathbb{C}\llbracket z-p\rrbracket$, $r\in\mathbb{C}\llbracket z-p\rrbracket$ and $\ln(z-p)f+r$ a solution of $\mathcal{H}_a(L)$ at $z=p$. Then 
\begin{enumerate}
 \item \[h=z^{1-a}(z-p)^{a+\mu-1}\sum_{m=0}^{\infty}\mathscr{B}(\mu+1+m,a-1)A_m(z-p)^m\] is a solution of $L$ at $z=p$.
\item $H_a^p(h)\ \widehat{=}\ f$.
\end{enumerate}
\end{lemma}

\begin{proof}
By Proposition \ref{Convlog} and Lemma \ref{HadInv}, the expression
\[g=(z-p)^{\mu+a}\sum_{m=0}^{\infty}\mathscr{B}(\mu+1+m,a)A_m(z-p)^m\] is a solution of $\prod_{k=1}^{m-1}(\vartheta-k)\prod_{j=0}^{m-1}(\vartheta-a-j)L^{z^{-a}}\vartheta$ at $z=p$.\\ 
As $p$ is no singularity of $\prod_{k=1}^{m-1}(\vartheta-k)\prod_{j=0}^{m-1}(\vartheta-a-j)$ and $\mu+a\not\in \mathbb{Z}$, we have $L^{z^{-a}}\vartheta(g)=0$. 

Thus
\begin{align*}
z\frac{d}{dz}g=z(z-p)^{\mu+a-1}\sum_{m=0}^{\infty}\mathscr{B}(\mu+1+m,a-1)A_mz^m
\end{align*}
is a solution of $L^{z^{-a}}$ and we obtain the first part of the statement.
Setting $h=z^{1-a}\frac{d}{dz} g$ we get
\begin{align*}
z^{a-1}h=\frac{d}{dz}g=(z-p)^{a+\mu-1}\sum_{m=0}^{\infty}\mathscr{B}(\mu+1+m,a-1)A_mz^m.
\end{align*}
Thus Proposition \ref{Loslos} yields
\begin{align*}
H_a^p(h)&\ \widehat{=}\ (z-p)^{\mu}\sum_{m=0}^{\infty}\mathscr{B}(\mu+a+m,1-a)\mathscr{B}(\mu+1+m,a-1)A_m(z-p)^m\\&\hg(z-p)^{\mu}\sum_{m=0}^{\infty}A_m(z-p)^m=f.
\end{align*}
\qed
\end{proof}

Rephrasing the lemma above, at a singular point $p\not\in\{0,\infty\}$ the special holomorphic solutions $f$ are those, which induce solutions of the form $\ln(z-p)f+r$, where $r\in\mathbb{C}\llbracket z-p\rrbracket$. In the geometric context solutions of this type turn out to be interesting as indicated in \cite[Appendix B]{Can} or \cite[Chapter 6]{ES}.

\section{Construction of Calabi-Yau operators}

In this section, we combine the results of the preceding sections to construct families of irreducible fuchsian differential operators inducing monodromy tuples of type $P_1$ and $P_2$. We will also compute special solutions of those operators at some of the singular points explicitly. Next, we investigate which of the operators constructed in the first part seem to be Calabi-Yau in the sense of \cite{AESZ}. As recently uncovered in \cite{Geemen}, unlike the definition of a Calabi-Yau operator given in \cite{AESZ}, there are families of Calabi-Yau threefolds, hence also Calabi-Yau operators, having no point of maximally unipotent monodromy. However, we restrict ourselves to the classical case of having such a point. In particular, the families $P_i$ for $i\geq 3$ cannot be induced by an operator corresponding to such a classical family. All operators we find using this method are covered by \cite[Appendix A]{AESZ}, but in most of the cases we are unfortunately not able to show, whether the operators are Calabi-Yau.

In the sequel, we will use the notations introduced in the preceding sections. Let furthermore $t=\frac{1}{z}$. As we have seen before, the construction of monodromy tuples of type $P_1$ and $P_2$ splits into four cases, each of which we cover by the preceding theorems. Furthermore, we only construct those operators $L$ for which zero is the only exponent at $z=0$ and choose the singular locus of $L$ to be $S=\{0,1,\infty\}$. We collect the remaining exponents $\lambda_{1,1},\dots,\lambda_{4,1}$ at $z=1$ and $\lambda_{1,\infty},\dots,\lambda_{4,\infty}$ at $z=\infty$ in its Riemann-scheme
\[\mathcal{R}(L)=\begin{Bmatrix}
0&1&\infty\\[0.1cm] \hline\\[-0.25cm]\begin{array}{c} 0
\\\noalign{\medskip}0
\\\noalign{\medskip}0
\\\noalign{\medskip}0
\\\noalign{\medskip} \end{array}&\begin{array}{c} \lambda_{1,1}
\\\noalign{\medskip}\lambda_{2,1}
\\\noalign{\medskip}\lambda_{3,1}
\\\noalign{\medskip}\lambda_{4,1}
\\\noalign{\medskip} \end{array}&\begin{array}{c} \lambda_{1,\infty}
\\\noalign{\medskip}\lambda_{2,\infty}
\\\noalign{\medskip}\lambda_{3,\infty}
\\\noalign{\medskip}\lambda_{4,\infty}
\\\noalign{\medskip} \end{array}
\end{Bmatrix}.\] In all occurring cases, the Jordan forms of the local monodromies can be read off directly from the Riemann scheme, as only repeated exponents turn out to induce logarithms. 
Proofs of those statements which can be obtained directly using the methods established before are omitted. For the sake of clarity, we frequently use well known hypergeometric identities as stated in \cite{Bailey} without any further comment. 

\begin{thm}[\textbf{The $P_1(4,10,4)$ case}]\label{P11Op}
Let $a,b\in\MQ\setminus\MZ$. A two parameter family of operators inducing monodromy tuples of type $P_1(4,10,4)$ is given by
\begin{align*}
 \mathcal{P}^{(a,b)}_1(4,10,4)&:=L_a\star_H L_{1-a}\star_H L_b\star_H L_{1-b}\\&=\vartheta^4-z(\vartheta+a)(\vartheta+1-a)(\vartheta+b)(\vartheta+1-b).
\end{align*}
The Riemann scheme reads
\[\mathcal{R}\left(\mathcal{P}^{(a,b)}_1(4,10,4)\right)=
\begin{Bmatrix}
0&1&\infty\\[0.1cm] \hline\\[-0.25cm]\begin{array}{c} 0
\\\noalign{\medskip}0
\\\noalign{\medskip}0
\\\noalign{\medskip}0
\\\noalign{\medskip} \end{array}&\begin{array}{c} 0
\\\noalign{\medskip}1
\\\noalign{\medskip}1
\\\noalign{\medskip}2
\\\noalign{\medskip} \end{array}&\begin{array}{c} a
\\\noalign{\medskip}1-a
\\\noalign{\medskip}b
\\\noalign{\medskip}1-b
\\\noalign{\medskip} \end{array}
\end{Bmatrix}.\]
Special solutions of this operator are $f=\sum_{m=0}^{\infty} A_mz^m$ at $z=0$, where \[A_m=\binom{a+m-1}{m}\binom{m-a}{m}\binom{b+m-1}{m}\binom{m-b}{m},\] $g=(z-1)\sum_{m=0}^{\infty}B_m(z-1)^m$, where \[B_m=\frac{b-1}{m+1}\binom{1+m-b}{m}\sum_{l=0}^m(-1)^l\binom{-b}{m-l}\frac{_{3}F_2\left(\genfrac{}{}{0pt}{0}{a,\ -l,\ 1-a}{1,\ b-l}\Bigg{\vert}\ 1\right)}{b-l-1}\] and $h_{\gamma}=t^\gamma\sum_{m=0}^{\infty}C^{(\gamma)}_mt^m$ at $z=\infty$, where $\gamma\in E=\{a,1-a,b,1-b\}$ and
\[C^{(\gamma)}_m=\mathscr{B}(\gamma+m,1-a)\mathscr{B}(\gamma+m,a)\mathscr{B}(\gamma+m,1-b)\mathscr{B}(\gamma+m,b).\] 
Moreover, $g$ is the conifold-period of $\mathcal{P}^{(a,b)}_1(4,10,4)$ at $z=1$, i.e. there is an $r\in (z-1)\mathbb{C}\llbracket z-1\rrbracket$ such that $\ln(z-1)g+r$ is a solution of  $\mathcal{P}^{(a,b)}_1(4,10,4)$ at $z=1$.
\end{thm}

\begin{proof}
It is clear that $L_a\star_H L_{1-a}\star_H L_b\star_H L_{1-b}$ induces a monodromy tuple of type $P_1(4,10,4)$. As in Example \ref{Hadaab}, we get \[L_a\star_H L_{1-a}\star_H L_b\star L_{1-b}=\mathcal{H}_{1-b}(\mathcal{H}_{b}(\mathcal{H}_{1-a}(L_a)))=\vartheta^4-z(\vartheta+a)(\vartheta+1-a)(\vartheta+b)(\vartheta+1-b).\]
The formulae for $A_m$, $B_m$ and $C^{(\gamma)}_m$ can be obtained directly using Proposition \ref{Convlog} and exchanging the roles of $a,1-a,b$ and $1-b$ freely. It remains to show, that $g$ is the conifold-period at $z=1$. As $e=1$ is an exponent of multiplicity two at $z=1$, the method of Frobenius yields a solution $\ln(z-1)\tilde{g}+r$ of $ \mathcal{P}^{(a,b)}_1(4,10,4)$ at $z=1$, where $\tilde{g}\in (z-1)\mathbb{C}\llbracket z-1\rrbracket$ and $r\in (z-1)\mathbb{C}\llbracket z-1\rrbracket$. Applying the first statement of Lemma \ref{coni} yields a solution $\omega\in (z-1)^{1-b}\mathbb{C}\llbracket z-1\rrbracket$ of $L_a\star_H L_{1-a}\star_H L_b$. As $1-b$ is the only exponent of $L_a\star_H L_{1-a}\star_H L_b$ at $z=1$ lying in $-b+\mathbb{Z}$, we have $\omega\hg H_b^1\left(H_{1-a}^1\left((1-z)^{-a}\right)\right)$. Applying the second statement of Lemma \ref{coni} yields
\[\tilde{g}\hg H^1_{1-b}(\omega)\hg H^1_{1-b}\left(H_b^1\left(H_{1-a}^1\left((1-z)^{-a}\right)\right)\right)\hg g.\] \qed
\end{proof}

\begin{thm}[The $P_1(4,8,4)$ case]\label{P12Op}
Let $a\in\MQ\setminus\left(\frac{1}{4}+\MZ\cup\frac{3}{4}+\mathbb{Z}\right)$ and $b\in\MQ\setminus\left(\frac{1}{4}+\MZ\cup\frac{3}{4}+\mathbb{Z}\right)$. A two parameter family of operators inducing monodromy tuples of type $P_1(4,4,8)$ is given by
\begin{align*}
 \mathcal{P}^{(a,b)}_1(4,8,4)&:=\Lambda^2\left(\left(L_{a+\frac{1}{4}}\star_H L_{\frac{1}{4}-a}\right)^{z^{\frac{1}{2}}}\star_HL_{b+\frac{3}{4}}\star_HL_{\frac{3}{4}-b}\right)^{z^{-\frac{3}{2}}}\star_H L_{\frac{3}{2}}\\&=64\,{\vartheta}^{4}+z \left( -128\,{\vartheta}^{4}-256\,{\vartheta}^{3
}+{\vartheta}^{2}(128({a}^{2}+{b}^{2})-304)\right)\\&\quad +z\left(\vartheta\,(128({a}^{2}+b^2)-176)+48(a^2+b^2)+256\,{a}^{2}{b}^{2}-39 \right)\\&\quad +64{z
}^{2} \left( a+1+\vartheta-b \right)  \left( a+1+\vartheta+b \right) 
 \left( a-1-\vartheta-b \right)  \left( -1+a-\vartheta+b \right).
\end{align*}
The Riemann scheme reads
\[\mathcal{R}\left(\mathcal{P}^{(a,b)}_1(4,8,4)\right)=
\begin{Bmatrix}
0&1&\infty\\[0.1cm] \hline\\[-0.25cm]\begin{array}{c} 0
\\\noalign{\medskip}0
\\\noalign{\medskip}0
\\\noalign{\medskip}0
\\\noalign{\medskip} \end{array}&\begin{array}{c} -\frac{1}{2}
\\\noalign{\medskip}0
\\\noalign{\medskip}1
\\\noalign{\medskip}\frac{3}{2}
\\\noalign{\medskip} \end{array}&\begin{array}{c} 1-a-b
\\\noalign{\medskip}1+a-b
\\\noalign{\medskip}1-a+b
\\\noalign{\medskip}1+a+b
\\\noalign{\medskip} \end{array}
\end{Bmatrix}.\]
Special solutions of this operator are given by $f=\sum_{m=0}^{\infty} A_mz^m$ at $z=0$ with \[A_m=\binom{\frac{1}{2}+m}{m}\sum_{k=0}^m\left(2k-\frac{1}{2}-m\right)\alpha\left(-\frac{1}{2},k\right)\alpha(0,m-k),\]
where \begin{align*}\alpha(\nu,m):=&\mathscr{B}\left(\frac{3}{4}+a+\nu+m,\frac{1}{4}-a\right)\mathscr{B}\left(\frac{3}{4}-a+\nu+m,\frac{1}{4}+a\right)\\&\mathscr{B}\left(\frac{3}{4}+b+\nu+m,\frac{3}{4}-b\right)\mathscr{B}\left(\frac{3}{4}-b+\nu+m,\frac{3}{4}+b\right)\end{align*} and $h_{(\mu,\nu)}=t^{\mu+\nu}\sum_{m=0}^{\infty}C^{(\mu,\nu)}_mt^{m}$  at $z=\infty$, where
\[C^{(\mu,\nu)}_m=\mathscr{B}\left(\nu+\mu+m,-\frac{1}{2}\right)\sum_{k=0}^{m}(2k+\mu-\nu-m)\delta(\mu,k)\delta(\nu,m-k),\] with 
 \begin{align*}\delta(\mu,k):=&\mathscr{B}\left(\mu-\frac{1}{4}+k,\frac{3}{4}-a\right)\mathscr{B}\left(\mu-\frac{1}{4}+k,\frac{3}{4}+a\right)\\&\mathscr{B}\left(\mu+\frac{1}{4}+k,\frac{1}{4}-b\right)\mathscr{B}\left(\mu+\frac{1}{4}+k,\frac{1}{4}+b\right),\end{align*} for $\mu\in\left\{\frac{1}{2}+a,\frac{1}{2}-a\right\}$ and $\nu\in\left\{\frac{1}{2}+b,\frac{1}{2}-b\right\}$.
\end{thm}

\begin{thm}[The $P_2(4,6,6)$ case]\label{P22Op}
Let $a,b\in\MQ\setminus\MZ$. A two parameter family of operators inducing monodromy tuples of type $P_2(4,6,6)$ is given by
\begin{align*}
\mathcal{P}^{(a,b)}_2(4,6,6)&:=\Sym^2\left(\left(L_{a}\star L_{b}\right)^{(1-z)^{\frac{1-a-b}{2}}}\right)\star_HL_{\frac{1}{2}}\\&=4\,{\vartheta}^{4}-2\,z \left( 2\,\vartheta+1 \right) ^{2} \left( {
\vartheta}^{2}+\vartheta+2\,ab-a+1-b \right)\\&\quad-{z}^{2} \left( 2\,
\vartheta+3 \right)  \left( 2\,\vartheta+1 \right)  \left( b-1-a-
\vartheta \right)  \left( b+1-a+\vartheta \right).
\end{align*}
The Riemann scheme reads
\[\mathcal{R}\left(\mathcal{P}^{(a,b)}_2(4,6,6)\right)=\begin{Bmatrix}
0&1&\infty\\[0.1cm] \hline\\[-0.25cm]\begin{array}{c} 0
\\\noalign{\medskip}0
\\\noalign{\medskip}0
\\\noalign{\medskip}0
\\\noalign{\medskip} \end{array}&\begin{array}{c} 0
\\\noalign{\medskip}1
\\\noalign{\medskip}-\frac{1}{2}+a+b
\\\noalign{\medskip}\frac{3}{2}-a-b
\\\noalign{\medskip} \end{array}&\begin{array}{c} \frac{1}{2}
\\\noalign{\medskip}\frac{3}{2}
\\\noalign{\medskip}1+a-b
\\\noalign{\medskip}1-a+b
\\\noalign{\medskip} \end{array}
\end{Bmatrix}.\] Special solutions of this operator are
 $f=\sum_{m=0}^{\infty} A_mz^m$ at $z=0$, where \[A_m=\binom{m-\frac{1}{2}}{m}\sum_{k=0}^m\binom{a+k-1}{k}\binom{b+k-1}{k}\binom{m-k-a}{m-k}\binom{m-k-b}{m-k},\] $g_{(a,b)}$ and $g_{(1-a,1-b)}$, where $ g_{(a,b)}=(z-1)^{\frac{3}{2}-a-b}\sum_{m=0}^{\infty}B^{(a,b)}_m(z-1)^m$, with \[B^{(a,b)}_m=\mathscr{B}\left(2-a-b+m,\frac{1}{2}\right)\sum_{l=0}^m\mathscr{B}\left(1-b+l,1-a\right)\alpha(l)\binom{-\frac{1}{2}}{m-l}\binom{a-1}{l}\] and \[\alpha(l)=\,_4F_3\left(\genfrac{}{}{0pt}{}{-l,\ 1-b,\ 1-a,\ a-1-l+b}{b-l,\ a-l,\ 2-a-b}
\Bigg{\vert}\ 1
\right) \] and $h_{(a,b)}$ and $h_{(1-a,1-b)}$ at $z=\infty$, where $h_{(a,b)}=t^{1-a+b}\sum_{m=0}^{\infty}C^{(a,b)}_mt^m$ with \[C^{(a,b)}_m=\mathscr{B}\left(1-a+b+m,\frac{1}{2}\right)\sum_{l=0}^m \delta(l)\delta(m-l)\] and 
\[\delta(l)=\, _{3}F_2\left(\genfrac{}{}{0pt}{0}{b,\ b,\ -l}{\frac{1}{2}-a+b,\ \frac{1}{2}(1+a+b)-l}\Bigg{\vert}\ 1\right)\binom{-\frac {1}{2}(1+a+b)+l}{l}.\] 
\end{thm}

\begin{thm}[The $P_2(4,6,8)$ case]\label{P21Op}
Let $a,b\in\MQ\setminus\MZ$. A two parameter family of operators inducing monodromy tuples of type $P_2(4,6,8)$ is given by
\begin{align*}
\mathcal{P}^{(a,b)}_2(4,6,8)&:=\left(L_a\star_H L_a\right)^{(1-z)^{1-a}}\star_H L_b\star_H L_{1-b}\\&=\vartheta^4-z(\vartheta+b)(\vartheta+1-b)(2\vartheta^2+2\vartheta+a^2-a+1)\\&\quad+z^2(\vartheta+b)(\vartheta+1-b)(\vartheta+b+1)(\vartheta+2-b).
\end{align*}
The Riemann scheme reads
\[\mathcal{R}\left(\mathcal{P}^{(a,b)}_2(4,6,8)\right)=\begin{Bmatrix}
0&1&\infty\\[0.1cm] \hline\\[-0.25cm]\begin{array}{c} 0
\\\noalign{\medskip}0
\\\noalign{\medskip}0
\\\noalign{\medskip}0
\\\noalign{\medskip} \end{array}&\begin{array}{c} 0
\\\noalign{\medskip}1
\\\noalign{\medskip}a
\\\noalign{\medskip}1-a
\\\noalign{\medskip} \end{array}&\begin{array}{c} b
\\\noalign{\medskip}1-b
\\\noalign{\medskip}1+b
\\\noalign{\medskip}2-b
\\\noalign{\medskip} \end{array}
\end{Bmatrix}.\] Special solutions of this operator are $f=\sum_{m=0}^{\infty} A_mz^m$ at $z=0$, where \[A_m=\binom{b+m-1}{m}\binom{m-b}{m}\sum_{k=0}^m\binom{a+m-k-1}{m-k}^2\binom{k-a}{k}\] and $g_{\gamma}=(z-1)^{\gamma}\sum_{m=0}^{\infty}B^{(\gamma)}_m(z-1)^m$, where \begin{align*}B^{(\gamma)}_m=\mathscr{B}(1+\gamma-b+m,b)\sum_{l=0}^m(-1)^l\alpha(l)\mathscr{B}(1-b+l,-\gamma)\binom{-b}{m-l}\binom{l-1+\gamma}{\gamma-1},\end{align*} with $\gamma\in\{a,1-a\}$, where
\[\alpha(l)=\,_3F_2\left(\genfrac{}{}{0pt}{0}{-l,\ \gamma,\ \gamma}{1+\gamma,\ b-l}\Bigg{\vert}\ 1\right).\]
\end{thm}

Now we investigate which of the operators constructed before are differential Calabi-Yau operators in the spirit of \cite{AESZ}.
We first recall the definition of those objects, which still is quite conjectural and state some of their basic properties. From the geometric point of view, the solutions of a Calabi-Yau operator of order $n$ should correspond to periods of a family of Calabi-Yau manifolds of dimension $n-1$ with Picard number one. In this sense, Calabi-Yau operators should be special Picard-Fuchs operators, which can't be defined from the differential algebraic point of view in a proper way yet. According to our definition, Calabi-Yau operators respect common conjectures for a differential operator to be Picard-Fuchs, see e.g. \cite{KonZa}. Some of the arithmetic conditions for a differential operator to be Calabi-Yau are basically motivated by approaches of mirror symmetry as discussed in \cite{Can}, but still seem to be quite mysterious.

\begin{defin}
For $n\geq 2$, an irreducible operator $L=\partial^n+\sum_{i=0}^{n-1}a_i\partial^i\in\MQ(z)[\partial]$ is called \textbf{Calabi-Yau operator} if it satisfies the following conditions.
\begin{enumerate}
\item [(CY-1)] The point $z=0$ is a regular singularity of $L$ and zero is the only exponent at this point.
\item[(CY-2)]  $L$ has a solution $y_0$ which is $N$-integral at $z=0$, i.e. at $z=0$ it is of the form \[y_0=1+\sum_{m=1}^{\infty}A_mz^m\in\MQ\llbracket z\rrbracket,\] with $N^mA_m\in\MZ$ for each $m\geq 1$ and a fixed $N\in\MN$.
\item[(CY-3)] We have \[L\alpha=\alpha L^{*}\] for a non trivial solution $\alpha$ of the differential equation $\omega'=-\frac{2}{n}a_{n-1}\omega$. Here \[L^{*}=\partial^n+\sum_{i=0}^{n-1}(-1)^{n+i}\partial^ia_i\in\mathbb{C}(z)[\partial]\] denotes the dual operator of $L$.
\item [(CY-4)] There is a solution $y_1$ linearly independent of $y_0$ given in (CY-3), such that the differential equation \[\omega'=\left(\frac{y_1}{y_0}\right)'\omega\] has a non trivial solution $q\in z+z^2\mathbb{Q}\llbracket z\rrbracket$ at $z=0$ which is $N$-integral. Such a solution is often called the \textbf{q-coordinate} or special coordinate of $L$ at $z=0$. 
\end{enumerate}
\end{defin}

By the construction done in Theorems \ref{P11Op}-\ref{P22Op}, we get

\begin{lemma}
Each of the operators $\mathcal{P}^{(a,b)}_1(4,10,4)$, $\mathcal{P}^{(a,b)}_1(4,8,4)$, $\mathcal{P}^{(a,b)}_2(4,6,8)$ and $\mathcal{P}^{(a,b)}_2(4,6,6)$ constructed in Theorem \ref{P11Op}-\ref{P22Op} fulfills the properties (CY-1)-(CY-3).
\end{lemma}

\begin{proof}
Property (CY-1) can be read off the corresponding Riemann scheme directly. Using \cite[Theorem I.4.3]{Dwork} and \cite[Formula II.4.6]{Dwork}, one shows that the unique solution at $z=0$ lying in $1+\mathbb{Q}\llbracket z\rrbracket$ of each operator is $N$-integral. Finally, condition (CY-3) can be obtained by a direct computation.\qed
\end{proof}

It remains to investigate which of the operators fulfill property (CY-4). Although there have recently been many improvements in the technique of showing this property, see e.g. \cite{Kratt1} and \cite{Kratt2}, we are in most of the cases not able to decide whether condition (CY-4) holds or not. Let us point out that for each operator constructed here it is also possible to compute a solution of the form $\ln(z)y_0+y_1$ by taking $\frac{d}{d\mu}y_0\Bigg{\vert}_{\mu=0}$ of the holomorphic solution $y_0=\sum_{\nu=\mu}^{\infty}f(\mu)z^{\mu}\Bigg{\vert}_{\mu=0}$ as it is described in \cite[Chapter 16]{Ince} but that we are often not able to check, whether the criterion \cite[Proposition 4.1]{Kratt1} holds or the series can be treated by a specialization of \cite[Theorem 2]{Kratt2}.

Our investigations lead to the following

\begin{conj}
An $\Sp_4(\mathbb{C})$-rigid tuple consisting
 of quasi-unipotent elements and having a maximally unipotent element is induced by a differential Calabi-Yau operator if and only if the elements of its second exterior power lie up to simultaneous conjugation in $\SO_5(\mathbb{Z})$. Furthermore, the inducing operator is unique. 
\end{conj}

In the sequel we state which of the cases in each of the families correspond to operators listed in \cite[Appendix A]{AESZ} and refer to the number of the operator stated there. Note that the operators constructed here have singular locus $\{0,1,\infty\}$, so we get the corresponding operators after having performed a transformation of the form $z\mapsto \lambda z$ with $\lambda\in\mathbb{Q}^{*}$, which leaves the properties (CY-1)-(CY-4) untouched and changes the singular locus to $\left\{0,\frac{1}{\lambda},\infty\right\}$. It is remarkable that after having performed the transformation the coefficients of the $q$-coordinate are minimal over $\MZ$, meaning that they are all lying in $\MZ$ and there is no $\alpha\in\MZ$ such that $\alpha^m$ divides the $m$-th coefficient for each $m\in\MN$. Furthermore for each series of operators the transformation can be done uniformly. Let therefore in the sequel for $a=\frac{r}{s}$, where $r\in\mathbb{Z}$ and $s\in\mathbb{N}$ are coprime, \[\beta\colon\MQ\setminus\{0\}\to\overline{\MZ},\ a\mapsto s\prod_{i=1}^ns_i^{\frac{1}{s_i-1}},\] where $s_1,\dots,s_n$ denote the distinct prime divisors of $s$. 

\begin{enumerate}
\item[(i)] \textbf{The $P_1(4,10,4)$ case:}\\
Having performed the transformation $z\mapsto\beta(a)^2\beta(b)^2z$, we get the following Calabi-Yau operators
\setlength{\extrarowheight}{0.5mm}
\begin{longtable}{|c||c|c|c|c|c|c|c|c|c|c|c|c|c|c|}\hline
\textbf{a} & $\frac{1}{2}$&$\frac{1}{2}$&$\frac{1}{2}$&$\frac{1}{2}$&$\frac{1}{3}$&$\frac{1}{3}$&$\frac{1}{3}$&$\frac{1}{4}$&$\frac{1}{4}$&$\frac{1}{6}$&$\frac{1}{5}$&$\frac{1}{8}$&$\frac{1}{10}$&$\frac{1}{12}$\\[0.5mm]\hline
\textbf{b}&$\frac{1}{2}$&$\frac{1}{3}$& $\frac{1}{4}$&$\frac{1}{6}$&$\frac{1}{3}$&$\frac{1}{4}$&$\frac{1}{6}$&$\frac{1}{4}$&$\frac{1}{6}$&$\frac{1}{6}$&$\frac{2}{5}$&$\frac{3}{8}$&$\frac{3}{10}$&$\frac{5}{12}$\\[0.5mm]\hline
\textbf{Number}&$3$&$5$&$6$&$14$&$4$&$11$&$8$&$10$&$12$&$13$&$1$&$7$&$2$&$9$\\[0.5mm]\hline
\end{longtable}

\item[(ii)] \textbf{The $P_1(4,8,4)$ case:}\\
To make our observations more transparent, we substitute $c=2a+\frac{1}{2}$ and $d=2b+\frac{1}{2}$. 
Having performed the transformation $z\mapsto 4\beta(c)^2\beta(d)^2z$, we get the following Calabi-Yau operators
\setlength{\extrarowheight}{0.5mm}
\begin{longtable}{|c||c|c|c|c|c|c|c|c|c|c|c|c|c|c|}\hline
\textbf{c}&$\frac{1}{2}$&$\frac{1}{2}$&$\frac{1}{2}$&$\frac{1}{2}$&$\frac{1}{3}$&$\frac{1}{3}$&$\frac{1}{3}$&$\frac{1}{4}$&$\frac{1}{4}$&$\frac{1}{6}$&$\frac{1}{5}$&$\frac{1}{8}$&$\frac{1}{10}$&$\frac{1}{12}$\\[0.5mm]\hline
\textbf{d}&$\frac{1}{2}$&$\frac{1}{3}$& $\frac{1}{4}$&$\frac{1}{6}$&$\frac{1}{3}$&$\frac{1}{4}$&$\frac{1}{6}$&$\frac{1}{4}$&$\frac{1}{6}$&$\frac{1}{6}$&$\frac{2}{5}$&$\frac{3}{8}$&$\frac{3}{10}$&$\frac{5}{12}$\\[0.5mm]\hline
\textbf{Number}&$\tilde{3}$&$\tilde{5}$&$\tilde{6}$&$\tilde{14}$&$\tilde{4}$&$\tilde{11}$&$\tilde{8}$&$\tilde{10}$&$\tilde{12}$&$\tilde{13}$&$\tilde{1}$&$\tilde{7}$&$\tilde{2}$&$\tilde{9}$\\[0.5mm]\hline
\end{longtable}
where the number $\widetilde{i}$ refers to the operators defined in \cite{Alm2}. As shown there, those operators are equivalent to $206-219$ in \cite{AESZ}
\newpage Note that the operator \[Q_2^{c,d}:=\left(\Lambda^2 \mathcal{P}_1^{c,d}(4,8,8)\right)^{z^{-1}(1-z)^{-\frac{3}{2}}}\] is of hypergeometric type. Its Riemann scheme reads
\[\mathcal{R}\left(Q_2^{c,d}\right)=
\begin{Bmatrix}
0&1&\infty\\[0.1cm] \hline\\[-0.25cm]\begin{array}{c} 0
\\\noalign{\medskip}0
\\\noalign{\medskip}0
\\\noalign{\medskip}0
\\\noalign{\medskip} 0
\\\noalign{\medskip}\end{array}&\begin{array}{c} 0
\\\noalign{\medskip}1
\\\noalign{\medskip}\frac{3}{2}
\\\noalign{\medskip}2
\\\noalign{\medskip}3
\\\noalign{\medskip} \end{array}&\begin{array}{c} \frac{1}{2}
\\\noalign{\medskip}c
\\\noalign{\medskip}d
\\\noalign{\medskip}1-c
\\\noalign{\medskip}1-d
\\\noalign{\medskip} \end{array}\end{Bmatrix}.\]

This family contains elements whose induced monodromy group is not a subgroup of $\Sp_4(\mathbb{Z})$. 

\item[(iii)] \textbf{The $P_2(4,6,6)$ case:}\\
Having performed the transformation $z\mapsto 4\beta(a)\beta(b)z$, we get the following Calabi-Yau operators
\setlength{\extrarowheight}{0.5mm}
\begin{longtable}{|c||c|c|c|c|c|c|c|c|}\hline
$\textbf{a}$&$\frac{1}{2}$&$\frac{1}{2}$&$\frac{1}{2}$&$\frac{1}{2}$&$\frac{1}{3}$&$\frac{1}{3}$&$\frac{1}{3}$&$\frac{1}{3}$\\[0.5mm]\hline
$\textbf{b}$&$\frac{1}{2}$&$\frac{1}{3}$&$\frac{1}{4}$&$\frac{1}{6}$&$\frac{1}{3}$&$\frac{2}{3}$&$\frac{1}{6}$&$\frac{5}{6}$\\[0.5mm]\hline
\textbf{Number}&$3^{*}$
&$-$
&$6^*$
&$14^{*}$
&$4^{*}$
&$4^{**}$
&$8^{*}$
&$8^{**}$
\\[0.5mm]\hline
\end{longtable}
\begin{longtable}{|c||c|c|c|c|c|c|c|c|}\hline
$\textbf{a}$&$\frac{1}{4}$&$\frac{1}{4}$&$\frac{1}{6}$&$\frac{1}{6}$&$\frac{1}{8}$&$\frac{1}{8}$&$\frac{1}{12}$&$\frac{1}{12}$\\[0.5mm]\hline
$\textbf{b}$&$\frac{1}{4}$&$\frac{3}{4}$&$\frac{1}{6}$&$\frac{5}{6}$&$\frac{3}{8}$&$\frac{5}{8}$&$\frac{5}{12}$&$\frac{7}{12}$\\[0.5mm]\hline
\textbf{Number}&$10^{*}$
&$10^{**}$
&$13^{*}$
&$13^{**}$
&$7^{*}$
&$7^{**}$
&$9^{*}$
&$9^{**}$
\\[0.5mm]\hline
\end{longtable}
The case $a=\frac{1}{2}$ and $b=\frac{1}{3}$ is not listed here, since the corresponding operator is $\Sym^3$ of a second order operator. 

\item[(iv)] \textbf{The $P_2(4,6,8)$ case:}\\
Having performed the transformation $z\mapsto\beta(a)^2\beta(b)^2z$, we get the following Calabi-Yau operators
\setlength{\extrarowheight}{0.5mm}
\begin{longtable}{|c||c|c|c|c|c|c|c|c|}\hline
$\textbf{a}$&$\frac{1}{2}$&$\frac{1}{2}$&$\frac{1}{2}$&$\frac{1}{2}$&$\frac{1}{3}$&$\frac{1}{3}$&$\frac{1}{3}$&$\frac{1}{3}$\\[0.5mm]\hline
$\textbf{b}$&$\frac{1}{2}$&$\frac{1}{3}$&$\frac{1}{4}$&$\frac{1}{6}$&$\frac{1}{2}$&$\frac{1}{3}$&$\frac{1}{4}$&$\frac{1}{6}$\\[0.5mm]\hline
\textbf{Number}&$111$&$110$&$30$&$112$&$141$&$142$&$196$&$143$\\[0.5mm]\hline
\end{longtable}
\begin{longtable}{|c||c|c|c|c|c|c|c|c|}\hline
$\textbf{a}$&$\frac{1}{4}$&$\frac{1}{4}$&$\frac{1}{4}$&$\frac{1}{4}$&$\frac{1}{6}$&$\frac{1}{6}$&$\frac{1}{6}$&$\frac{1}{6}$\\[0.5mm]\hline
$\textbf{b}$&$\frac{1}{2}$&$\frac{1}{3}$&$\frac{1}{4}$&$\frac{1}{6}$&$\frac{1}{2}$&$\frac{1}{3}$&$\frac{1}{4}$&$\frac{1}{6}$\\[0.5mm]\hline
\textbf{Number}&$189$&$194$&$197$&$199$&$190$&$195$&$198$&$61$\\[0.5mm]\hline
\end{longtable}
\end{enumerate}

\appendix


\begin{appendix}
\section{Appendix: Subgroup structure of the $\Sp_4(\CC)$}

We give an overview of the maximal irreducible subgroups in $\Sp_4(\CC)$
and their behaviour under taking the exterior product.

\begin{lemma}\label{subgr}
 The maximal semisimple connected 
 subgroups of $\Sp_4(\CC)$ are contained in one of the
 following classes.
 \begin{enumerate}
  \item $(\Sp_2(\CC) \times \Sp_2(\CC)).2,$
   \item $ \GL_2(\CC).2 \cong \Sp_2(\CC) \ten \GO_2(\CC)$.
   \item $\Sym^3\SL_2(\CC),$
 \end{enumerate}
 where $.2$ denotes a group extension of degree $2$.
\end{lemma}

\begin{proof}
 A maximal connected semisimple subgroup $G$ of $\Sp_4(\CC)$
 can be written as a product $G=G_1\cdots G_r$ of simple
 groups $G_i$. Hence $G_i$ is either a torus or $\Sp_2(\CC)$. 
 Since the Lie-rank of $\Sp_4(\CC)$ is two we get $r\leq 2.$
 This gives the claim, cf. \cite[Chap. 1]{Ca85}.
\qed\end{proof}

\begin{cor}\label{la2}
 Two classes of the maximal irreducibles subgroups in $\Sp_4(\CC)$  become reducible in $\SO_5(\CC)$ taking their antisymmetric square.
 \begin{eqnarray*}
 \Lambda^2 (\Sp_2(\CC) \ten\GO_2(\CC) )&
 = & \langle (A,B) \in \GO_3(\CC)  \times  \GO_2(\CC)\mid \det(A)\det(B)=1\rangle\\
 \Lambda^2 (\Sp_2(\CC) \times \Sp_2(\CC)).2 ) & = &\GO_4(\CC),
\end{eqnarray*}
 where $ \GO_4(\CC)$ is naturally embedded into $\SO_5(\CC)$.   
\end{cor}

\begin{proof}
 The claims follow from the  identities.
 \begin{eqnarray*}\Lambda^2 (V_1\ten V_2) &= & \Lambda^2 V_1\ten \Sym^2 V_2\; \oplus \Sym^2 V_1\ten \Lambda^2 V_2, \\
  \Lambda^2(V_1 \oplus V_2)& = &\Lambda^2 (V_1) \;\oplus V_1 \ten V_2 \oplus \;\Lambda^2  V_2. 
 \end{eqnarray*}

\qed\end{proof}

\begin{cor}\label{Zar}
 Let $H$ be a irreducible proper subgroup
 of $\Sp_4(\CC)$.
 Then the following hold (up to conjugation of $H$).
 \begin{enumerate}
  \item  If $H$ contains a unipotent element with Jordan form $\J(4)$ then
         $H \subseteq \Sym^3 \Sp_2(\CC)$.

  \item  If $H$ contains a transvection then
         $H \subseteq (\Sp_2(\CC) \times \Sp_2(\CC)).2$.
  \item   If all non trivial unipotent elements in $H$ have the
          Jordan form $(\J(2),\J(2))$
         then
         $H \subseteq \SL_2\ten \GO_2(\CC)$.
 \end{enumerate}
\end{cor}

\begin{proof}
 The claims follow from Lemma~\ref{subgr}.
\qed\end{proof}
\end{appendix}
\newpage

\providecommand{\bysame}{\leavevmode\hbox to3em{\hrulefill}\thinspace}
\providecommand{\MR}{\relax\ifhmode\unskip\space\fi MR }
\providecommand{\MRhref}[2]{%
  \href{http://www.ams.org/mathscinet-getitem?mr=#1}{#2}
}
\providecommand{\href}[2]{#2}


\begin{thebibliography}{CdlOGP98}

\bibitem[AESZ10]{AESZ}
Gert Almkvist, {Christian van} Enckevort, {Duco van} Straten, and Wadim
  Zudilin, \emph{Tables of {C}alabi--{Y}au equations}, 2010, preprint.
  http://arxiv.org/abs/math/0507430.

\bibitem[Alm06]{Alm2}
Gert Almkvist, \emph{Calabi--{Y}au differential equations of degree 2 and 3 and
  {Y}ifan {Y}ang's pullback}, 2006, preprint.
  http://arxiv.org/abs/math/0612215.

\bibitem[And89]{Andre}
Yves Andr\'e, \emph{{G}-functions and geometry}, Aspects of {M}athematics, vol.
  E13, Vieweg und Sohn, 1989.

\bibitem[Bai35]{Bailey}
Wilfrid~N. Bailey, \emph{Generalized {H}ypergeometric {S}eries}, Cambrigde
  {U}niversity {P}ress, 1935.

\bibitem[BD79]{BalDwo79}
Francesco Baldassarri and Bernard Dwork, \emph{On second order linear
  differential equations with algebraic solutions}, Amer. J. Math. \textbf{101}
  (1979), 42--76.

\bibitem[BH89]{BH90}
Frits Beukers and Gert Heckman, \emph{Monodromy for the hypergeometric function
  {$\sb nF\sb {n-1}$}}, Inventiones Mathematicae \textbf{95} (1989), no.~2,
  325--354.

\bibitem[Car85]{Ca85}
Roger~W. Carter, \emph{Finite groups of {L}ie type. {C}onjugacy classes and
  complex characters}, {Pure and Applied Mathematics}, John Wiley \& Sons,
  Inc., 1985.

\bibitem[CdlOGP98]{Can}
Philip Candelas, Xenia~C. de~la Ossa, Paul~S. Green, and Linda Parkes, \emph{A
  pair of {C}alabi--{Y}au manifolds as an exactly soluble superconformal field
  theory}, {M}irror {S}ymmetrie {I}, Studies in advanced mathematics, vol.~9,
  American Mathematical Society, 1998, pp.~31--95.

\bibitem[Del70]{Del}
Pierre Deligne, \emph{Equations diff\'erentielles a points singuliers
  r\'eguliers}, Lecture Notes in Mathematics, vol. 163, Springer-Verlag,
  Heidelberg, 1970.

\bibitem[DGS94]{Dwork}
Bernard Dwork, Giovanni Gerotto, and Francis~J. Sullivan, \emph{An introduction
  to {G-}functions}, Annals of mathematical studies, vol. 133, Princeton
  {University} {Press}, 1994.

\bibitem[DR00]{DR99}
Michael Dettweiler and Stefan Reiter, \emph{An algorithm of {Katz} and its
  application to the inverse {Galois} problem}, Journal of Symbolic Computation
  \textbf{30} (2000), no.~6, 761--798.

\bibitem[DR07]{DR07}
\bysame, \emph{Middle convolution of {Fuchsian} systems and the construction of
  rigid differential systems}, J. Algebra \textbf{318} (2007), no.~1, 1--24.

\bibitem[GvG10]{Geemen}
Alice Garbagnati and Bert van Geemen, \emph{Examples of {C}alabi-{Y}au
  threefolds parametrised by {S}himura varieties}, 2010, preprint.
  http://arxiv.org/abs/math/1005.0478.

\bibitem[IKSY91]{IKShY}
Katsunori Iwasaki, Hironobu Kimura, Shun Shimomura, and Masaaki Yoshida,
  \emph{From {G}auss to {Painlev\'e}--{A} modern theory of special functions},
  Aspects of Mathematics, Vieweg, Braunschweig, 1991.

\bibitem[Inc56]{Ince}
Edward~L. Ince, \emph{Ordinary differential equations}, Dover, London, 1956.

\bibitem[Kat96]{katz96}
Nicholas~M. Katz, \emph{Rigid local systems}, Annals of mathematical studies,
  vol. 139, Princeton {University} {Press}, 1996.

\bibitem[KR08]{Kratt2}
Christian Krattenthaler and Tanguy Rivoal, \emph{Multivariate p-adic formal
  congruences and integrality of {T}aylor coefficients of mirror maps}, 2008,
  preprint. http://arxiv.org/abs/math/0804.3049.

\bibitem[KR10]{Kratt1}
\bysame, \emph{On the integrality of {T}aylor coefficients of mirror maps},
  Duke Math. J. \textbf{151} (2010), 175--218.

\bibitem[KZ01]{KonZa}
Maxim Kontsevich and Don~B. Zagier, \emph{Periods}, Mathematics
  unlimited---2001 and beyond, Springer, Berlin, 2001, pp.~771--808.

\bibitem[Lev61]{Levelt}
Antonius~H.M. Levelt, \emph{Hypergeometric functions}, 1961, thesis,
  {U}niversity of {A}msterdam.

\bibitem[PS02]{Put}
{Marius van der} Put and Michael~F. Singer, \emph{Galois theory of linear
  differential equations}, 2nd ed., Grundlehren der Mathematischen
  Wissenschaften, vol. 328, Springer-Verlag, Berlin, 2002.

\bibitem[Sco77]{Scott}
Leonard~L. Scott, \emph{Matrices and cohomology}, Ann. {M}ath. (1977),
  473--492.

\bibitem[Sim90]{Simpson2}
Carlos Simpson, \emph{Transcendental aspects of the {Riemann-Hilbert}
  correspondence}, Illinois journal of mathematics \textbf{34} (1990), no.~2,
  368--391.

\bibitem[Sim92]{Simpson92}
\bysame, \emph{Higgs bundles and local systems}, Publ. Math. {IHES} \textbf{75}
  (1992), 5--95.

\bibitem[Sin96]{Singer}
Michael~F. Singer, \emph{Testing reducibility of linear differential operators:
  a group theoretic perspective}, Applicable Algebra in Engineering,
  Communication and Computing \textbf{7} (1996), no.~2, 77--104.

\bibitem[SV99]{SV}
Karl Strambach and Helmut V{\"o}lklein, \emph{On linearly rigid tuples}, {J.
  Reine Angew. Math.} (1999), no.~510, 57--62.

\bibitem[vEvS04]{ES}
Christian van {Enckevort} and Duco van {Straten}, \emph{{M}onodromy
  calculations of fourth order equations of {C}alabi--{Y}au type}, 2004,
  preprint. http://arxiv.org/abs/math/0412539.

\end{thebibliography}
\end{document}